\documentclass[11pt, a4paper,reqno]{amsart}
\pdfoutput=1
\usepackage{tikz,enumitem,rotating,latexsym,bm,stmaryrd,caption,}
\usetikzlibrary{positioning,intersections,decorations,shadings}
\usetikzlibrary{shapes}
\usetikzlibrary{arrows}
\usetikzlibrary{patterns}
\usepackage{tkz-euclide}
\usetikzlibrary{calc}
\usetikzlibrary{shapes.geometric}
\tikzset{snake it/.style={decorate, decoration=snake}}
\tikzset{zigzag/.style={decorate, decoration=zigzag}}

\definecolor{ao(english)}{rgb}{0.0, 0.5, 0.0}
\definecolor{darkgreen}{rgb}{0.0, 0.5, 0.0}

\newcommand{\sgnlamu}{{\sf sgn}(\la,\mu)}
\newcommand{\sgnnula}{{\sf sgn}(\nu,\la)}
\newcommand{\sgnalla}{{\sf sgn}(\alpha,\la)} 
\newcommand{\sgnnual}{{\sf sgn}(\nu,\alpha)} 
\newcommand{\sgnalnu}{{\sf sgn}(\alpha,\nu)}
\newcommand{\sgnnumu}{{\sf sgn}(\nu,\mu)}

\pgfdeclareverticalshading{rainbow}{100bp}
{color(0bp)=(orange);color(25bp)=(orange);
color(50bp)=(magenta); color(75bp)=(cyan);
 color(100bp)=(cyan)}

\usetikzlibrary{decorations.pathmorphing}

	\definecolor{eng}{rgb}{0.0, 0.5, 0.0}
\definecolor{apple}{rgb}{0.55, 0.71, 0.0}
\definecolor{cadmium}{rgb}{0.0, 0.42, 0.24}
\definecolor{darkspringgreen}{rgb}{0.09, 0.45, 0.27}
\definecolor{amethyst}{rgb}{0.6, 0.4, 0.8}
\definecolor{ao}{rgb}{0.0, 0.0, 1.0}
\definecolor{atomictangerine}{rgb}{1.0, 0.6, 0.4}
\definecolor{carmine}{rgb}{0.59, 0.0, 0.09}

\definecolor{toggle}{rgb}{1.0, 0.94, 0.96}

\usepackage[normalem]{ulem}

 \newcommand{\csigma}{{{{\color{magenta}\bm\sigma}}}}
 \newcommand{\ctau}{{{{\color{cyan}\bm\tau}}}}

 \newcommand{\al}{{{{\color{magenta}\bm \sigma}}}}
 \newcommand{\crho}{{{{ \color{green!80!black}\bm\rho}}}}

\newcommand{\w}{{\underline{w}}}
\newcommand{\vvv}{{\underline{v} }}

\newcommand{\y}{{\underline{y}}}
\newcommand{\x}{{\underline{x}}}

\def\down{\vee}
\def\up{\wedge}

\usepackage{standalone}

\usepackage{etex}
 
\tikzset{
  variable line width/.style={
    every variable line width/.append style={#1},
    to path={%
      \pgfextra{%
        \draw[every variable line width/.try,line width=\pgfkeysvalueof{/tikz/thickness}] (\tikztostart) -- (\tikztotarget);
      }%
      (\tikztotarget)
    },
  },
  thickness/.initial=0.6pt,
  every variable line width/.style={line cap=round, line join=round},
}

\usepackage{todonotes}

\usepackage{tikz}
\usetikzlibrary{matrix,  intersections, calc, decorations.pathreplacing} 

\usepackage{tikz-cd}

\newlength{\superthick}
\newlength{\cornerradius}
\tikzstyle{corner}=[rounded corners=\cornerradius]
\tikzstyle{dot}=[circle, inner sep=0pt, minimum size=4.8pt]
\tikzstyle{string}=[line width=\superthick]
\tikzstyle{std}=[string,dash pattern=on 0.9pt off 0.9pt]
\definecolor{realcyan}{rgb}{0,1,1}

 \usepackage{amsmath,amsthm,amsfonts,amssymb,mathrsfs,pb-diagram}
\usepackage[
bookmarks=true,colorlinks=true,linktoc=page,citecolor=green!80!black,linkcolor=green!80!black,urlcolor=green!80!black]{hyperref}
\usepackage{caption}
\usepackage{lipsum,wasysym}
\usepackage{mathtools}
\usepackage[a4paper,margin=1in]{geometry}
\usepackage{cleveref}
 \usepackage{amsmath}
\mathchardef\mhyphen="2D
\usepackage{color}
\usepackage{xcolor}
\usepackage{ifthen}
\usepackage{sidecap}   
\definecolor{mediumblue}{rgb}{0.0, 0.0, 0.8}

\newcommand{\mptn}{{\mathscr{P}_{m,n}}}

\renewcommand{\geq}{\geqslant}
\renewcommand{\leq}{\leqslant}

\tikzset{wei/.style= 
{red,double=red,double
distance=0.5pt}}

 \newcommand{\bet}{{{\color{cyan}\boldsymbol \tau}}}

\tikzset{wei2/.style={red,double=red,double
distance=0.5pt}}

\allowdisplaybreaks
\numberwithin{equation}{section}
\usepackage{scalefnt}

\newtheorem{thm}{Theorem}[section]
\newtheorem{cor}[thm]{Corollary}

\newtheorem{lem}[thm]{Lemma}

\newtheorem{prop}[thm]{Proposition}

\newtheorem*{prop*}{Proposition}
\newtheorem*{thmA}{Theorem A}

\newtheorem*{thmB}{Theorem B}
\newtheorem*{thmC}{Theorem C}
\newtheorem*{cor*}{Corollary}

\newtheorem*{conj*}{Conjecture D}

\newtheorem*{conj1*}{Conjecture B}
\newtheorem*{Acknowledgements*}{Acknowledgements}

\theoremstyle{rmk}

\theoremstyle{defn}
\newtheorem{rmk}[thm]{Remark}
\newtheorem{defn}[thm]{Definition}
\newtheorem{eg}[thm]{Example}

\newcommand{\Std}{{\rm Std}}

\newcommand{\la}{\lambda}

\newcommand{\sts}{\mathsf{s}}  
\newcommand{\stt}{\mathsf{t}}  
\newcommand{\stu}{\mathsf{u}}  

\newcommand{\ZZ}{{\mathbb Z}}
\newcommand{\NN}{{\mathbb N}}

\let\<=\langle
\let\>=\rangle

\tikzset{
ultra thin/.style= {line width=0.05pt},
very thin/.style=  {line width=0.2pt},
thin/.style=       {line width=0.1pt},
semithick/.style=  {line width=0.6pt},
thick/.style=      {line width=0.8pt},
very thick/.style= {line width=1.2pt},
ultra thick/.style={line width=1.6pt}
}

\crefname{ques}{Question}{Questions}
\crefname{defn}{Definition}{Definitions}
\crefname{thm}{Theorem}{Theorems}
\crefname{prop}{Proposition}{Propositions}
\crefname{lem}{Lemma}{Lemmas}
\crefname{cor}{Corollary}{Corollaries}
\crefname{conj}{Conjecture}{Conjectures}
\crefname{section}{Section}{Sections}
\crefname{subsection}{Subsection}{Subsections}
\crefname{eg}{Example}{Examples}
\crefname{figure}{Figure}{Figures}
\crefname{rem}{Remark}{Remarks}
\crefname{rmk}{Remark}{Remarks}
\crefname{equation}{equation}{equation}

\Crefname{ques}{Question}{Questions}
\Crefname{defn}{Definition}{Definitions}
\Crefname{thm}{Theorem}{Theorems}
\Crefname{prop}{Proposition}{Propositions}
\Crefname{lem}{Lemma}{Lemmas}
\Crefname{cor}{Corollary}{Corollaries}
\Crefname{conj}{Conjecture}{Conjectures}
\Crefname{section}{Section}{Sections}
\Crefname{subsection}{Subsection}{Subsections}
\Crefname{eg}{Example}{Examples}
\Crefname{figure}{Figure}{Figures}
\Crefname{rem}{Remark}{Remarks}
\Crefname{rmk}{Remark}{Remarks}

\usepackage[hang,flushmargin]{footmisc}

\begin{document}

 \title[Hecke categories and Khovanov arc algebras]{
Quiver presentations and isomorphisms \\  of Hecke categories and Khovanov arc algebras 
  }

 \author{Chris Bowman}
       \address{Department of Mathematics, 
University of York, Heslington, York,  UK}
\email{Chris.Bowman-Scargill@york.ac.uk}
  
 \author{Maud De Visscher}
	\address{Department of Mathematics, City, University of London,   London, UK}
\email{Maud.DeVisscher.1@city.ac.uk}

		\author{Amit  Hazi}
	 
     \address{Department of Mathematics, 
University of York, Heslington, York,  UK}
\email{Amit.Hazi@york.ac.uk}

		\author{ Catharina Stroppel}
 \address{ Mathematical Institute, Endenicher Allee 60, 53115 Bonn}
 \email{Stroppel@math.uni-bonn.de}
 
 \maketitle

 \vspace{-0.5cm}
\begin{abstract}
We prove that the extended  Khovanov arc algebras are isomorphic to the basic algebras of 
anti-spherical Hecke categories for maximal   parabolics of symmetric groups.  
We present these algebras by  quiver and relations     and provide the full submodule lattices of Verma modules. 
  \end{abstract}

\section{Introduction}

A fundamental notion of categorical Lie theory is that 
of uniqueness.
  If a pair of 2-categorical objects share the  same underlying  Kazhdan--Lusztig theory, then they ``should  be  
   the same"   --- a striking example of this is given by the
$\ZZ$-graded    algebra isomorphisms between KLR algebras and diagrammatic Soergel bimodules
 conjectured in \cite{MR4100120} and proven in \cite{cell4us2}. 
Our starting point for this paper is the   simplest family of  ($p$-)Kazhdan--Lusztig polynomials --- those given by  oriented Temperley--Lieb diagrams, or equivalently, Dyck tilings. These combinatorial objects underly the (extended) Khovanov arc algebras \cite{MR2918294} and the 
  anti-spherical 
Hecke categories associated to maximal    parabolics   of   symmetric groups \cite{compan}.
 The first theorem  of this paper    explains this coincidence via  an  explicit and elementary 
 $\ZZ$-graded      algebra isomorphism (see Theorem~A below) in the spirit of \cite{MR4100120,cell4us2}.
 
 By a theorem of Gabriel, any basic algebra over a field is isomorphic to the path algebra of its ${\rm Ext}$-quiver, modulo relations.  Such presentations are one of the ``holy grails" of representation theory
--- they essentially provide complete information about the structure of an algebra.  
  Our second main theorem of this paper provides such a presentation for the basic algebras of 
anti-spherical  Hecke categories of 
  maximal    parabolics   of   symmetric groups (Theorem~B). 
  We use this presentation to obtain complete submodule lattices for the Verma modules   (Theorem~C).
  
The results of this paper are mostly self-contained, with elementary proofs which work over any integral domain $\Bbbk$ containing  $i \in \Bbbk$ such that  $i^2=-1$. 

\subsection{The isomorphism theorem}

 The (extended) Khovanov arc algebras $\mathcal{K} _{m,n}$ for $m,n\in \NN$ 
have their origins in 2-dimensional topological quantum field theory and their first applications were in  categorical knot theory  \cite{MR1740682,MR2521250}.  
These algebras have subsequently been studied from the point of view of 
Springer varieties \cite{MR1928174},
their cohomological and representation theoretic structure  \cite{MR2600694,MR2781018,MR2955190,MR2881300,BarWang},
and symplectic geometry \cite{MR4422212}
and  have further inspired much generalisation: from the Temperley--Lieb setting to web diagrams \cite{web1,web2,web3,web4}
and also  from the ``even" setting to    ``super"   \cite{odd1} and  ``odd" settings \cite{odd2}, as well as to the orthosymplectic case \cite{MR3518556,MR3563723,MR3644792}.
In summary, these algebras form the prototype for  knot-theoretic categorification, we   refer to Stroppel's 2022 ICM address for more details 
\cite{ICM1}.  

On the other hand, anti-spherical Hecke categories of parabolic Coxeter systems $(W,P)$ provide   the universal setting for studying the interaction between Kazhdan--Lusztig theory and categorical Lie theory --- they have formed the crux of the resolutions of the Jantzen, Lusztig, and Kazhdan--Lusztig positivity conjectures \cite{MR3689943,w13,MR3245013} and control much of the representation theory of algebraic groups and braid groups \cite{cell4us2,BNS2,MR3868004,MR4454848,ELpaper}. 
We refer to   Williamson's 2018 ICM address for a more complete history and the geometric motivation for their study 
\cite{MR3966750}.

Our first main theorem bridges the gap between these two distinct categorical worlds: 

\begin{thmA}
The extended Khovanov arc algebras $\mathcal{K} _{m,n}$  are isomorphic (as   $\ZZ$-graded $\Bbbk$-algebras)
to the basic algebras $\mathcal{H}_{m,n} $ of the  anti-spherical  Hecke categories associated to the maximal parabolics of symmetric groups $(S_{m+n},S_m \times S_n)$ for all $m,n\in\NN$.
\end{thmA}

Given the vast generalisations of Khovanov arc algebras  (in particular to the super world!) 
and of these anti-spherical Hecke categories (to all parabolic Coxeter systems) 
we hope that our main theorem will be the starting point of much reciprocal study of  these two worlds.

\subsection{Quiver and relations for   Hecke categories  }
It is well-known that ($p$-)Kazhdan--Lusztig polynomials encode a great deal of character-theoretic 
and indeed cohomological information about Verma modules (particularly if one puts certain restrictions on $p\geq 0$).  If the algebra is Koszul (as is the case for our algebras) we further know that the $p$-Kazhdan--Lusztig polynomials carry complete information about the radical layers of   indecomposable   projective and Verma modules.  
Given the almost ridiculous level of detail these polynomials encode, 
 it is natural to ask {\em``what are the limits to what $p$-Kazhdan--Lusztig combinatorics can tell us about the  structure of the Hecke category?"}

The starting point of this  paper is to delve deep  into the Dyck/Temperley--Lieb combinatorics for $p$-Kazhdan--Lusztig polynomials, which was initiated in 
    \cite{MR2918294,compan}.  There is a  wealth of extra, richer combinatorial information which can be encoded into the Dyck tilings underlying these
    $p$-Kazhdan--Lusztig 
     polynomials.  
     Instead of looking only at the sets of Dyck tilings 
     (which enumerate the $p$-Kazhdan--Lusztig 
     polynomials) we look at the  relationships for passing 
     between
     these Dyck tilings.  
     In fact, this ``meta-Kazhdan--Lusztig combinatorics"  
      is sufficiently rich as to completely determine the full structure of our Hecke categories:
       
 \begin{thmB}
The $\Bbbk$-algebra  $ \mathcal{H}_{m,n} $  admits a quadratic presentation as the path algebra of the ``Dyck quiver" $\mathscr{D}_{m,n}$ of \cref{quiverdefn}
 modulo ``Dyck-combinatorial relations" \eqref{rel1} to \eqref{adjacent}.
If  $\Bbbk$ is a field,  then  the ${\rm Ext}$-quiver of  
$\mathcal{H}_{m,n} 
$  is isomorphic  to $\mathscr{D}_{m,n}$ 
and  this gives a presentation of the algebra by quiver and relations.
  \end{thmB}

 In a nutshell,  the power of Theorem B is that it allows us to understand not only the 
{\em graded composition series} of standard and projective modules 
 (the purview of classical Kazhdan--Lusztig combinatorics) but the {\em explicit 
 extensions interrelating  these composition factors} 
 (in terms of      meta-Kazhdan--Lusztig combinatorics).
   In essence, Theorem B  provides complete information about the structure of the anti-spherical Hecke categories 
of $(S_{m+n},S_m \times S_n)$ for $m,n\in\NN$.  
We reap some of the fruits of Theorem B  by providing an incredibly  elementary description of the  full submodule lattices of Verma modules:

   \begin{thmC}
  Let $\Bbbk$ be a field. 
 The submodule lattice of the Verma module $
 \Delta _{m,n}(\la)$ can be calculated  in terms of  the combinatorics of  Dyck tilings;  moreover this   lattice is  independent of the characteristic of $\Bbbk$.  
 
 \end{thmC}

An example is depicted  in \cref{submodules}, below. 

Specialising to the case that $\Bbbk$ is a field and putting Theorems A and B together, we obtain a 
conceptually simpler proof of the results of \cite[Section 2]{BarWang} (which makes use of the Koszulity of these algebras over a field, which is the main result of  \cite{MR2600694}).


%
%
%
%
%
%

\begin{figure}[ht!]
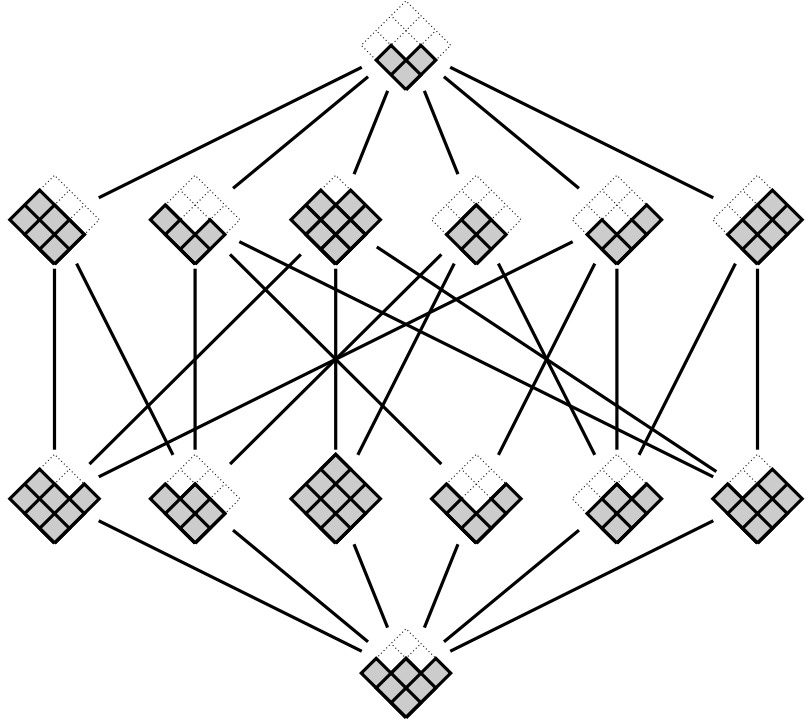


$$

 $$

 \caption{The full submodule lattice of the Verma module $\Delta_{m,n}(2,1)$ for  
 $\mathcal{H}_{m,n} $ for $m,n=3$ and $\Bbbk$ any field. 
 We represent each simple module by the corresponding partition (in Russian notation) and highlight the $3\times 3$ rectangle in which the partition exists. 
 This module has simple head $L_{3,3}(2,1)$ and simple socle $L_{3,3}(3,2,1)$. 
  Each edge connects a pair of partitions which differ by adding or removing a single Dyck path. }
 \label{submodules}
 
 \end{figure}

  \subsection{Structure of the paper}
In Section   2 we recall  the necessary combinatorics of oriented Temperley--Lieb diagrams and   $p$-Kazhdan--Lusztig polynomials from \cite{MR2600694,compan2}.  
In Section 3 we recall  the  extended Khovanov arc algebras  and the basic algebras of the Hecke categories which will be of interest in this paper.  
In 
Sections   4  and 5 we  develop the Dyck path combinatorics
 and 
lift this to the level of   generators and bases of 
the basic algebras of the Hecke categories.   
 In   Section 6 we take a short detour to discuss  the notion of {\em dilation} 
 for our diagram algebras, which will simplify the main proofs significantly.  
   In Sections 7  we prove Theorem B of this paper, by 
  lifting the Dyck combinatorics 
   to the level of a of $\mathcal{H}_{m,n} $ over an integral domain $\Bbbk$; 
 we then recast  this presentation in terms of the quotient of the path  algebra of the ${\rm Ext}$-quiver in the case that   $\Bbbk$ is a field.
 In Section   8,
 we prove Theorem C.
    Finally, in Section 9 we use Theorem B to prove the isomorphism of Theorem A.

\begin{Acknowledgements*}
The first and third authors were funded by  EPSRC grant 
EP/V00090X/1.  
\end{Acknowledgements*}
   
  \section{The combinatorics of Kazhdan-Lusztig polynomials}

     We begin by reviewing and unifying the combinatorics of Khovanov arc algebras 
 \cite{MR2600694,MR2781018,MR2955190,MR2881300} and the Hecke categories of interest in this paper \cite{compan,compan2}.
     
    \subsection{Cosets, weights and partitions}
 
Let $S_n$ denote the symmetric group of degree $n$. Throughout this paper, we will work with the parabolic Coxeter system $(W,P) = (S_{m+n}, S_m \times S_n)$. We label the simple reflections with the slightly unusual subscripts $s_i, \, -m+1 \leq i \leq n-1$ so that $P = \langle s_i \, | \, i\neq 0\rangle \leq W$. We view $W$ as the group of permutations of the $n+m$ points on a horizontal strip numbered by the half integers $i\pm \tfrac{1}{2}$ where the simple reflection $s_i$ swaps the points $i-\tfrac{1}{2}$ and $i+\tfrac{1}{2}$ and fixes every other point. 
The right cosets of $P$ in $W$ can then be identified by labelled horizontal strips called {\sf weights}, where each point $i\pm \tfrac{1}{2}$ is labelled by either $\up$ or $\down$ in such a way that the total number of $\up$ is equal to $m$ (and so the total number of $\down$ is equal to $n$). Specifically, the trivial coset $P$ is represented by the weight with negative points labelled by  $\up$  and positive points labelled by $\down$. The other cosets are obtained by permuting the labels of the identity weight. An example is given on the left hand side of Figure \ref{typeAtiling-long}.

 \begin{figure}[ht!]
 $$

$$
\caption{
We depict the 
 weight $\varnothing$ along the bottom of the diagram, 
  the weight of  $\la$ along the top of the diagram,    
  and the coset 
  $\color{magenta}s_0
  \color{green!60!black}s_{-1}
    \color{orange}s_{-2}  
      \color{lime!80!black}s_{-3}
          \color{gray}s_{-4}
                    \color{cyan}s_1
          \color{magenta}s_0
  \color{green!60!black}s_{-1}
    \color{orange}s_{-2}
                     \color{violet}s_2
                    \color{cyan}s_1
                    \color{pink}s_3        \color{violet}s_2
$.
  }
\label{typeAtiling-long}
\end{figure}

We denote by ${^PW}$ the set of minimal length right coset representative of $P$ in $W$.  Recall that an element $w\in {^PW}$ precisely when every reduced expression of $w$ starts with $s_0$. This implies that $w$ must be fully commutative, that is no reduced expression for $w$ contains a subword of the form $s_i s_{i\pm 1}s_i$ for some $i$. It follows that the elements of ${^PW}$ can also be represented by partitions that fit into an $m\times n$ rectangle. An example of the correspondence between $w\in {^PW}$, its weight diagram and the associated partition is illustrated in Figure \ref{typeAtiling-long}. 

Formally, a {\sf partition}    $\lambda $ of $\ell$  is defined to be a weakly decreasing  sequence   of non-negative integers $\lambda = (\lambda_1, \lambda_2, \ldots )$ which  sum to $\ell$.  We call $\ell (\lambda) := \ell = \sum_{i}\lambda_i$ the length of the partition $\lambda$.
We define the Young diagram of a partition to be the collection of tiles 
$$[\la]=\{[r,c] \mid 1\leq c \leq \la_r\}$$
depicted in Russian style with rows at $135^\circ$ and columns at $45^\circ$.  We identify a partition with its Young diagram.
We let $\la^t$ denote the transpose partition given by reflection 
of the Russian Young diagram through the vertical axis.  
 Given  $m,n\in \NN$ we let  ${\mathscr P}_{m,n}$ denote the set of all partitions which fit into an $m\times n$ rectangle, that is 
$${\mathscr P}_{m,n}= \{ \la \mid \la_1\leq m, \la_1^t \leq n\}.$$
For $\lambda \in {\mathscr P}_{m,n}$, the $x$-coordinate of a tile $[r,c] \in \lambda$ is   equal to   $r-c \in \{-m+1, -m+2 , \ldots , n-2, n-1\}$ and we define this $x$-coordinate to be  the {\sf content} (or ``colour")   of the tile and we  write ${\sf cont}[r,c]=r-c$.  

For a partition $\lambda$ of $\ell$, we define a standard tableau of shape $\lambda$ to be a bijection $\stt$ from the set of tiles of $\lambda$ to the set $\{1, 2, \ldots , \ell\}$ such that for each $1 \leq k \leq \ell$, the set of tiles $\stt^{-1}(\{1, \ldots , k\})$ is a partition. We can view $\stt$ as a filling of the tiles of $\lambda$ by the number $1$ to $\ell$ such that the numbers increase along rows and columns. We denote by ${\rm Std}(\lambda)$ the set of all standard tableaux of shape $\lambda$. 
There is one particular standard tableau  that we will be using throughout this paper which is defined as follows. We let $\stt_{\lambda}\in {\rm Std}(\lambda)$ denote the tableau in which we fill the tiles of $\lambda$ according to increasing $y$-coordinate and then refine according to increasing $x$-coordinate.  An example is depicted in \cref{super}.
For $\la$ a partition of $\ell$, we define the {\sf content sequence} of  $\sts \in \Std(\la)$   to be the element of $\ZZ^\ell$ given by reading the contents of the boxes in order.

Under the correspondence between ${^PW}$ and ${\mathscr P}_{m,n}$ described above, the content $i$ of each tile of $\lambda \in {\mathscr P}_{m,n}$ corresponds to the subscript of a simple reflection $s_i$. So we will often refer to the simple reflection $\ctau = s_i$ as the content of the tile. Moreover, standard tableaux $\stt\in {\rm Std}(\lambda)$ correspond precisely to reduced expressions $\lambda = s_{i_1}s_{i_2} \ldots s_{i_{\ell}}$ where $\stt^{-1}(j) = [r_j, c_j]$ with ${\sf cont}[r_j, c_j] = i_j$ for each $1\leq j \leq \ell$.
The {\sf Bruhat order} on ${^PW}$ becomes simply the inclusion of the (Young diagrams) of partitions in ${\mathscr P}_{m,n}$.

Given $\lambda \in {\mathscr P}_{m,n}$, we define the set ${\rm Add}(\lambda)$ to be the set of all tiles $[r,c]\notin \lambda$ such that $\lambda \cup [r,c]\in {\mathscr P}_{m,n}$. Similarly, we define the set ${\rm Rem}(\lambda)$ to be the set of all tiles $[r,c]\in \lambda$ such that $\lambda \setminus [r,c] \in {\mathscr P}_{m,n}$. Note that a partition $\lambda$ has at most one addable or removable tile of each content. So for $[r,c]\in {\rm Add}(\lambda)$ (respectively $[r,c]\in {\rm Rem}(\lambda)$) with $\ctau = s_{r-c}$ we write $\lambda + \ctau$ (respectively $\lambda - \ctau$) for $\lambda \cup [r,c]$ (respectively $\lambda \setminus [r,c]$).

 \begin{figure}[ht!]
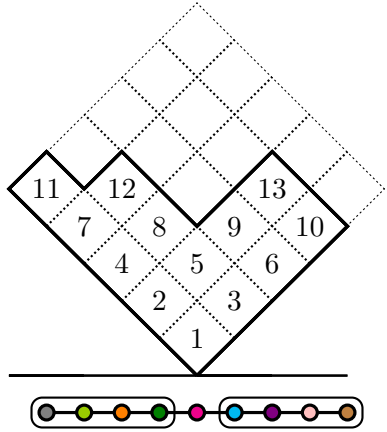

 $$

$$
 \caption{
A tableau,  $\stt_{ \la}$, of shape $ (5,4,2^2)$.
  }
 \label{super}
 \end{figure}

We have seen how to pass from a coset to a weight diagram and a partition. We now explain how to go directly from a weight diagram to a partition. 
Read the labels of a weight diagram from left to right. Starting at the left most corner of the $m\times n$ rectangle, take a north-easterly step for each $\down$ and a south-easterly step for each $\up$. We end up at the rightmost corner of the rectangle, having traced out the ``northern perimeter" of the Russian Young diagram.  
In particular, the identity coset corresponds to the weight diagram labelled by $m$ $\up$'s followed by $n$ $\down$'s, tracing the perimeter of the empty partition $\varnothing$. Throughout the paper, we will identify minimal coset representative with their weight diagrams and partitions.

    \subsection{Oriented Temperley--Lieb diagrams and Kazhdan--Lusztig polynomials} 
 
 The following definitions come from \cite{MR2918294}.
    
 \begin{defn} 
\begin{itemize}[leftmargin=*]
 \item To each weight $\lambda$ we associate a {\sf cup diagram} $\underline{\lambda}$ and a {\sf cap diagram} $\overline{\lambda}$. To construct $\underline{\lambda}$,  
repeatedly find a pair of vertices labeled $\down$ $\up$ in order from left to right that are neighbours in the sense that there are only vertices already joined by cups in between. Join these new vertices together with a cup. Then repeat the process until there are no more such $\down$ $\up$ pairs. Finally draw rays down at all the remaining $\up$ and $\down$ vertices. The cap diagram $\overline{\lambda}$ is obtained by flipping $\underline{\lambda}$ horizontally. We stress that the vertices of the cup and cap diagrams  are not labeled. 
\item Let $\lambda$ and $\mu$ be weights. We can glue $\underline{\mu}$ under $\lambda$ to obtain a new diagram $\underline{\mu}\lambda$. We say that $\underline{\mu}\lambda$ is  {\sf oriented} if (i) the vertices at the ends of each cup in $\underline{\mu}$ are labelled by exactly one $\down$  and one $\up$ in the weight $\lambda$ and (ii) it is impossible to find two rays in $\underline{\mu}$ whose top vertices are labeled $\down$ $\up$  in that order from left to right in the weight $\lambda$. 
Similarly, we obtain a new diagram $\lambda \overline{\mu}$ by gluing $\overline{\mu}$ on top of $\lambda$. We say that $\lambda \overline{\mu}$ is oriented if $\underline{\mu} \lambda$ is oriented. 
\item Let $\lambda$, $\mu$ be weights such that $\underline{\mu}\lambda$ is oriented. We set the {\sf degree} of the diagram $\underline{\mu}\lambda$ (respectively $\lambda \overline{\mu}$) to the the number of clockwise oriented cups (respectively caps) in the diagram. 
\item Let $\la, \mu ,\nu$ be weights such that $\underline{\mu}\la$ and $\la \overline{\nu}$ are oriented. Then  we form a new diagram $\underline{\mu}\la\overline{\nu}$ by gluing $\underline{\mu}$ under and $\overline{\nu}$ on top of $\la$. We set ${\rm deg}(\underline{\mu}\la \overline{\nu}) = {\rm deg}(\underline{\mu}\la)+{\rm deg}(\la \overline{\nu})$.
\end{itemize}
\end{defn}

An example is provided in Figure \ref{figure4}.

For the purposes of this paper, for $p\geq 0$, we can define the $p$-Kazhdan--Lusztig polynomials of type $(W,P) = (S_{n+m},S_m \times S_n)$ as follows.  
For $\la,  \mu \in \mptn$ we set
$$
{^p}n_{\la,\mu}= 
\begin{cases}
q^{\deg(\underline{\mu} \la)}		&\text{if $ \underline{\mu} \la $ is oriented}\\
0						&\text{otherwise.}
\end{cases}
$$
We refer to \cite[Theorem 7.3]{compan2} and \cite[Theorem A]{compan} for a justification of this definition and to \cite{MR2918294} for the origins of this combinatorics.

  \begin{figure}[ht!]
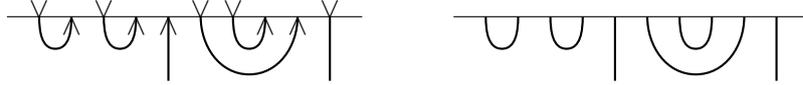

  $$   
$$
\caption{The construction of the cup diagram $\underline{\la}$ for $\la=(5,4,2^2)$. See also \cref{XXX}.}
\label{figure4}
	\end{figure}
  
 It is clear that for a fixed $\mu\in \mptn$, the diagram $\underline{\mu}\lambda$ is oriented if and only if the weight $\lambda$ is obtained from the weight $\mu$ by swapping the labels on some of the pairs of vertices connected by a cup in $\underline{\mu}$. Moreover, in this case the degree of $\underline{\mu}\la$ is precisely the number of such swapped pairs. 
See \cref{cref-it2} for an example of a cup diagram of degree 8.

\begin{figure}[ht!]
$$
     \begin{tikzpicture}  [yscale=0.6 ,xscale=0.6 ]

 \draw[lime!80!black,  thick]  (6,0) to [out=-90,in=180] (6.5,-0.6) to [out=0,in=-90] (7,0)  ; 
 \draw[   thick]  (5,0) to [out=-90,in=180] (6.5,-0.9) to [out=0,in=-90] (8,0)  ; 
 
\draw[    orange, thick]  (4,0) to [out=-90,in=180] (6.5,-1.2) to [out=0,in=-90] (9,0)  ;

\draw[ , thick]  (3,0) to [out=-90,in=180] (9.5,-1.3-0.2-0.2-0.2) to [out=0,in=-90] (16,0)  ; 

\draw[violet,     thick]  (2,0) to [out=-90,in=180] (9.5,-1.5-0.3-0.2-0.2) to [out=0,in=-90] (17,0)  ; 

\draw[ magenta, thick]  (1,0) to [out=-90,in=180] (9.5,-1.7-0.2-0.2-0.2-0.2) to [out=0,in=-90] (18,0)  ;

 \draw[    thick]  (012,0) to [out=-90,in=180] (12.5,-0.6) to [out=0,in=-90] (13,0)  ;

 \draw[   cyan, thick]  (011,0) to [out=-90,in=180] (12.5,-0.9) to [out=0,in=-90] (14,0)  ;

 \draw[  gray,  thick]  (010,0) to [out=-90,in=180] (12.5,-1.2) to [out=0,in=-90] (15,0)  ;

 \draw[  pink,  thick]  (20,0) to [out=-90,in=180] (20.5,-0.6) to [out=0,in=-90] (21,0)  ; 

 \draw[ brown,  thick]  (19,0) to [out=-90,in=180] (20.5,-0.9) to [out=0,in=-90] (22,0)  ;

 \draw(-0.5,0)--++(0:23);

\foreach \i in {0,3,5,7,9,12,14,15,17,18,21,22}
 {
  \draw (\i,0.11) node {$\scriptstyle\down$};}

\foreach \i in {1,2,4,6,8,10,11,13,16,19,20}
{  \draw(\i,-0.11) node {$\scriptstyle\up$};} 

\draw[thick](0,0)--++(-90:2);
\draw[densely dotted,thick](0,-2)--++(-90:0.3);

\end{tikzpicture}  
$$
\caption{The diagram $\underline{\mu}\la$ for 
 $\la=(11,9,8,7,6,4,3^2,2^2)$ and 
  $ 
\mu=  (11^7,8^3,2^2)
  $.
 We have coloured the positive degree  arcs, which should be compared (colour-wise!) with the  Dyck paths in   \cref{cref-it}. }
\label{cref-it2}
\end{figure}
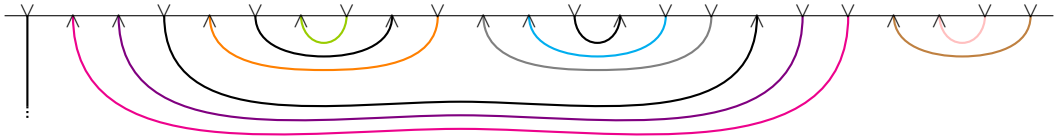

There is an alternative construction of the cup diagram $\underline{\mu}$ as the top half of a Temperley--Lieb diagram $e_\mu$. 
An $(m+n)$- {\sf Temperley--Lieb  diagram }is a rectangular
	frame with, in our case, $m+n$ vertices along   the top and $m+n$  along the bottom 
	which are paired-off by   non-crossing strands.  
	We refer to a strand connecting a top and bottom vertex as a {\sf propagating strand}.  We refer to any strand connecting two top vertices as a  {\sf cup} and any strand  connecting two bottom vertices as a {\sf cap}.   
For $\mu\in \mptn$, the Temperley--Lieb diagram $e_\mu$ is obtained by starting from the partition $\mu$ and taking the product of the \lq Temperley--Lieb generator' in each of its tiles. The cup diagram $\underline{\mu}$ is then simply the top half of the Temperley--Lieb diagram $e_\mu$.	This is illustrated in \cref{XXX}.  For more details, see \cite{compan2}.

\begin{figure}[ht!]
$$

$$
\caption{The element $e_\la$ for $\la=(5,4,2^2)$.}
\label{XXX}
\end{figure}

Now, let $d$ be any $(m+n)$-Temperley--Lieb diagram, and $\la$, $\mu$ be weights, then we can form a new diagram $\lambda d \mu$ by gluing the weight $\lambda$ under $d$ and $\mu$ on top of $d$. We say that $\lambda d \mu$ is an {\sf oriented Temperley--Lieb diagram} if for each  propagating strand in $d$
 its two  vertices
 are either both  $\down$ symbols or both $\up$ symbols and for each cup or cap in $d$ 
 its two  vertices
 consist of precisely one  $\down$ symbol and one   $\up$ symbol.

 It is easy see that for $\lambda, \mu\in \mptn$ we have that $\underline{\mu}\la$ is oriented if and only if $\varnothing e_\mu \lambda$ is an oriented Temperley--Lieb diagram. Throughout the paper, we will always view the oriented Temperley--Lieb diagrams $\varnothing e_\mu \la$ on the Young diagram of the partition $\mu$ as illustrated in Figure \ref{typeAtiling-long2}.

It was shown in \cite{compan2} that we can define an graded algebra structure on the space spanned by all oriented Temperley--Lieb diagrams. A crucial ingredient in the construction of the light leaves basis for the Hecke category (see Definition 5.1 below) comes from writing each oriented Temperley--Lieb diagram $\varnothing e_\mu \la$ as a product of generators for this algebra. We will not need the explicit presentation for this graded algebra here and will instead view this product of generators as an \lq oriented tableau' $\stt_\mu^\la$. This oriented tableau $\stt_\mu^\la$ is obtained by assigning to each tile of  the partition $\mu$ not only a number between $1$ and $\ell(\mu)$ defined by the tableau $\stt_\mu$ but also one of four possible orientations determined by the weight $\lambda$.  

\begin{defn} Let $\la, \mu \in \mptn$ such that  $\underline{\mu}\la$ is oriented. Draw the Temperley--Lieb diagram $e_\mu$ on the tiles of $\mu$ as in Figure \ref{XXX}. Gluing $\varnothing$ and $\la$ on the bottom and top of the diagram respectively defines one of four possible orientations on each tile of $\mu$. 
We define the orientation label  of a tile 
 as follows:
  $$ 
 $$

We then define the oriented tableau $\stt_\mu^\la$ to be the map which assign to each tile $[r,c]\in \mu$ a pair $(k, x)$ where $k=\stt_\mu ([r,c])$ and $x\in \{ 1, s , f ,sf\}$ is the orientation label of the tile $[r,c]$.

\end{defn}

 \begin{figure}[ht!]
 $$
  %
$$
\caption{
We depict the 
 weight $\varnothing$ along the bottom of the diagram, 
  the weight of  $\la$ along the top of the diagram,    
  and the coset 
  $\color{magenta}s_0
  \color{green!60!black}s_{-1}
    \color{orange}s_{-2}  
      \color{lime!80!black}s_{-3}
          \color{gray}s_{-4}
                    \color{cyan}s_1
          \color{magenta}s_0
  \color{green!60!black}s_{-1}
    \color{orange}s_{-2}
                     \color{violet}s_2
                    \color{cyan}s_1
                    \color{pink}s_3        \color{violet}s_2
$.
   The first diagram is   ${  e} _{ {(5,4,2^2)}   }$.  The latter diagram (of degree 2) is obtained by reorienting the  arc connecting the first to the fourth northern vertex 
   and the  arc connecting the sixth and ninth northern vertices.  
  }
\label{typeAtiling-long2}
\end{figure}

\section{The diagrammatic algebras}
 
 We now recall the construction of the two protagonists of this paper: the Hecke categories associated to $(S_{m+n}, S_m\times S_n)$ and the (extended) Khovanov arc algebras.

\subsection{The Hecke categories }
\label{coulda}

 We denote by $S = \{s_i \, : \, -m+1 \leq i \leq n-1\}$ the set of simple reflections. To simplify notations, for $\csigma = s_i, \ctau=s_j \in S$ we write $|\csigma - \ctau | :=i-j$. So we have $|\csigma - \ctau|>1$ precisely when $\csigma \ctau = \ctau \csigma$ and $|\csigma - \ctau| = 1$ precisely when  $\csigma \ctau \csigma= \ctau \csigma \ctau$.

\medskip

 We define the {\sf  Soergel generators} to be the framed graphs 
$$
{\sf 1}_{\emptyset } =
\begin{minipage}{1.5cm} 
\end{minipage} 
    $$  
associated to any pair $\csigma, \ctau\in S$ such that $|\csigma-\ctau| >1$. 

     Pictorially, we define the duals of these generators to be the graphs obtained by reflection through their horizontal axes.  Non-pictorially, we simply swap the sub- and superscripts.  We sometimes denote duality by $\ast$.  For example, the dual of the fork generator is pictured as follows
$$
  {\sf fork}^{\csigma\csigma}_{\csigma}=  
  \begin{minipage}{1.8cm}  
 %
\end{minipage}.$$ 
We define the northern/southern reading word of a Soergel generator (or its dual) to be word in the alphabet $S$ obtained by reading the colours of the northern/southern edge of the frame respectively.  
Given   two (dual) Soergel generators $D$ and $D'$ we define $D\otimes D'$ to be the diagram obtained by horizontal concatenation (and we extend this linearly).  The northern/southern colour sequence of $D\otimes D'$ is  the concatenation of those of $D$ and $D'$ ordered from left to right.    
Given any two (dual) Soergel generators, we define their product $D\circ D'$ (or simply $DD'$) to be the vertical concatenation of $D$ on top of $D'$ if the southern reading word of $D$ is equal to 
 the northern reading word of $D'$ and  zero otherwise.   
We define a Soergel graph to be any diagram obtained by vertical and horizontal concatenation of the generators and their duals.   

 Given $\csigma$, we define  the corresponding ``barbell", ``dork" and ``gap'' diagrams to be the elements  
 $${\sf bar}(\csigma)=  
  {\sf spot}_ \csigma^\emptyset
   {\sf spot}^ \csigma_\emptyset
   \qquad
   {\sf dork}^{\csigma\csigma}_{\csigma\csigma}= 
      {\sf fork}^{\csigma\csigma}_{ \csigma}
            {\sf fork}_{\csigma\csigma}^{ \csigma} \qquad {\sf gap}(\csigma) = {\sf spot}^\csigma_\emptyset {\sf spot}^\emptyset_\csigma.$$
  Given $\sts ,\stt\in\Std(\la)$ for $\la \in \mptn$, we write 
${\sf braid}^{\sts}_{\stt }$ for the  permutation mapping one tableau to the other, with strands coloured according to the contents.   For example 
$${\sf braid}^{\sts}_{\stt }=\begin{minipage}{3cm}\begin{tikzpicture}[scale=1.5]
\draw[densely dotted, rounded corners] (-0.25,0) rectangle (1.75,1);
\draw[magenta,line width=0.08cm](0,0)--++(90:1);
\draw[cyan,line width=0.08cm](1,0) to   (1.5,1);

\draw[darkgreen,line width=0.08cm](0.5,0)--(0.5,1);
\draw[orange,line width=0.08cm](1.5,0)-- (1,1);
\end{tikzpicture}\end{minipage} 
\qquad
{\sf braid}^{\stt}_{\stu }=\begin{minipage}{3cm}\begin{tikzpicture}[scale=1.5]
\draw[densely dotted, rounded corners] (-0.25,0) rectangle (1.75,1);
\draw[magenta,line width=0.08cm](0,0)--++(90:1);
\draw[cyan,line width=0.08cm](0.5,0) to  (1,1);

\draw[darkgreen,line width=0.08cm](1,0)--(0.5,1);
\draw[orange,line width=0.08cm](1.5,0)-- (1.5,1);
\end{tikzpicture}\end{minipage} 
\qquad 
{\sf braid}^{\sts}_{\stu}=\begin{minipage}{3cm}\begin{tikzpicture}[scale=1.5]
\draw[densely dotted, rounded corners] (-0.25,0) rectangle (1.75,1);
\draw[magenta,line width=0.08cm](0,0)--++(90:1);
\draw[cyan,line width=0.08cm](0.5,0) to [out=90,in=-90] (1.5,1);

\draw[darkgreen,line width=0.08cm](1,0)--(0.5,1);
\draw[orange,line width=0.08cm](1.5,0)-- (1,1);
\end{tikzpicture}\end{minipage} 
$$
for the three  $\sts,\stt,\stu \in \Std(3,1)$.
We sometimes drop the explicit tableaux and simply record the underlying content sequence.  


  \begin{rmk} 
The cyclotomic quotients of (anti-spherical) Hecke categories are small categories  
with finite-dimensional morphism spaces given by the light leaves basis of \cite{MR3555156,antiLW}.     
Working with such a category is equivalent to working with a locally unital algebra, as defined in \cite[Section 2.2]{BSBS}, see  \cite[Remark 2.3]{BSBS}.  
   Throughout this paper we will work in the latter setting.  
   The reader who prefers to think of categories can equivalently phrase everything in this paper  in terms of categories and representations of categories.

   \end{rmk}

We are  ready to define the first diagrammatic algebra of interest in this paper.  The following simplification of the presentation of the Hecke category is made using \cite[Theorem 2.1]{compan}

 \begin{defn}\label{easydoesit}   \renewcommand{\vvv}{{\underline{w} }} 
\renewcommand{\w}{{\underline{x}}}
\renewcommand{\x}{{\underline{y}}}
\renewcommand{\y}{{\underline{z}}}
\newcommand{\zz}{{\underline{w}}}
 
 We define  $\mathscr{H}_{m,n}   $ to be the locally-unital  associative $\Bbbk$-algebra spanned by all  Soergel-graphs  with multiplication given by $\circ$-concatenation modulo the following local relations and their vertical and horizontal flips.  Firstly,  for any   $ \csigma, \ctau \in  S$  we have the idempotent relations 
\begin{align*}
{\sf 1}_{\csigma} {\sf 1}_{\ctau}& =\delta_{\csigma,\ctau}{\sf 1}_{\csigma} & {\sf 1}_{\emptyset} {\sf 1}_{\csigma} & =0 & {\sf 1}_{\emptyset}^2& ={\sf 1}_{\emptyset}\\
{\sf 1}_{\emptyset} {\sf spot}_{\csigma}^\emptyset {\sf 1}_{\csigma}& ={\sf spot}_{\csigma}^{\emptyset} & {\sf 1}_{\csigma} {\sf fork}_{\csigma\csigma}^{\csigma} {\sf 1}_{\csigma\csigma}& ={\sf fork}_{\csigma\csigma}^{\csigma} &
{\sf 1}_{\ctau\csigma}  {\sf  braid}_{\csigma\ctau}^{\ctau\csigma}
 {\sf 1}_{\csigma\ctau}  & ={\sf  braid}_{\csigma\ctau}^{\ctau\csigma} 
\end{align*}
where the final relation holds for all ordered pairs  $( \csigma,\ctau)\in S^2$  with  $|\csigma- \ctau|>1$.  
For each  $\csigma \in  S $  we have   fork-spot contraction, the double-fork, and circle-annihilation  
 relations:  
\begin{align*} 
({\sf spot}_\csigma^\emptyset \otimes {\sf 1}_\csigma){\sf fork}^{\csigma\csigma}_{\csigma}
=
{\sf 1}_{\csigma}
 \qquad 
  ({\sf 1}_\csigma\otimes {\sf fork}_{\csigma\csigma}^{ \csigma} )
({\sf fork}^{\csigma\csigma}_{\csigma}\otimes {\sf 1}_{\csigma})
=
{\sf fork}^{\csigma\csigma}_{\csigma}
{\sf fork}^{\csigma}_{ \csigma\csigma}
\qquad  {\sf fork}_{\csigma\csigma}^{\csigma}
 {\sf fork}^{\csigma\csigma}_{\csigma}=0 
\end{align*}
   For   $(\csigma, \ctau ,\crho)\in S^3$ 
with   $|\csigma- \crho|,
|   \crho- \ctau| , |\csigma- \ctau|>1$,  we have the commutation relations 
 \begin{align*} 
 {\sf spot}^{ \csigma }_{\emptyset} \otimes {\sf 1}_{\crho} = 
{\sf braid}^{\csigma\crho}_{\crho\csigma}
( {\sf 1}_{\crho} \otimes {\sf spot}^{ \csigma }_{\emptyset} )
 \qquad
(  {\sf fork}^{ \csigma  \csigma}_{ \csigma} \otimes {\sf 1}_{\crho}  ) 
{\sf braid}^{\csigma\crho}_{\crho\csigma}
=
{\sf braid}^{\csigma\csigma\crho}_{\crho\csigma\csigma}
 (     {\sf 1}_{\crho} \otimes {\sf fork}^{ \csigma  \csigma}_{ \csigma} ) 
   \end{align*}
 \begin{align*} 
 {\sf braid}^{\csigma\ctau\crho}_{ \csigma\crho\ctau }
 {\sf braid}^ { \csigma\crho\ctau } _{  \crho\csigma\ctau } 
 {\sf braid}^ {  \crho\csigma\ctau }  _{  \crho \ctau\csigma }  
=
 {\sf braid}^{\csigma\ctau\crho}_{\ctau \csigma\crho  }
 {\sf braid}^ {\ctau \csigma\crho  } _{  \ctau  \crho  \csigma} 
 {\sf braid}^ {   \ctau  \crho  \csigma }  _{  \crho \ctau\csigma }  
     \end{align*}
    For $\csigma,\ctau \in S$ with $|\csigma- \ctau|=1$ we have the one and two colour Demazure relations:
\begin{align*}
     {\sf bar}(\csigma)\otimes {\sf 1}_\csigma
     +
 {\sf 1}_\csigma\otimes           {\sf bar}(\csigma) 
& =2 {\sf gap}(\csigma)
\\      {\sf bar}(\ctau)\otimes {\sf 1}_\csigma
     -
 {\sf 1}_\csigma\otimes           {\sf bar}(\ctau) 
& =		 {\sf 1}_\csigma\otimes           {\sf bar}(\csigma) 	- {\sf gap}(\csigma)
  \end{align*}
     and the null braid relation
\begin{align*} 
 {\sf 1}_{\csigma\ctau\csigma} + 
 ({\sf 1}_\csigma \otimes {\sf spot}^\ctau _\emptyset \otimes  {\sf 1}_\csigma )
 {\sf dork}^{\csigma\csigma}_{\csigma\csigma}
 ( {\sf 1}_\csigma\otimes 
  {\sf spot}_\ctau ^\emptyset \otimes  {\sf 1}_\csigma )
 =0
\end{align*}
   Further,  we require the interchange law and the monoidal unit relation 
$$
 \big( {\sf D  }_1 \otimes   {\sf D}_2   \big)\circ  
\big(  {\sf D}_3  \otimes {\sf D }_4 \big)
=
 ({\sf D}_1 \circ   {\sf D_3}) \otimes ({\sf D}_2 \circ  {\sf D}_4)
\qquad {\sf 1}_{\emptyset} \otimes {\sf D}_1={\sf D}_1={\sf D}_1 \otimes {\sf 1}_{\emptyset}
$$
for all diagrams ${\sf D}_1,{\sf D}_2,{\sf D}_3,{\sf D}_4$.  
 Finally, we require the   non-local cyclotomic   relations  
 \begin{align*} 
{\sf 1}_\csigma \otimes D=0
  \qquad 
  {\sf bar}(\ctau) \otimes D=0
\end{align*}
 for all ${\color{magenta}s_0}\neq \csigma \in S $, $\ctau \in S$  and $D$ any diagram.   

\smallskip

   We also define the idempotent truncation  $${\sf 1}_{m,n}= \sum_{\la\in\mptn } {\sf 1}_{\stt_\la}\qquad 
    \mathcal{H}_{m,n}= {\sf 1}_{m,n} \mathscr{H}_{m,n} {\sf 1}_{m,n}$$

 \end{defn}

  \begin{rmk} In \cite[Theorem 4.21]{compan} the algebra $   \mathcal{H}_{m,n}$ is shown to be the basic algebra of the anti-spherical  Hecke category 
 for $W= {S}_{m+n}$ the finite symmetric group and $P
 = {S}_{m}\times S_n
 \leq W$ a maximal parabolic subgroup.  
 \end{rmk}

   \begin{rmk} 
The algebras $\mathscr{H}_{m,n}$ and $\mathcal{H}_{m,n}$ can be equipped with 
 a $\mathbb Z$-grading which preserves the duality~$\ast$.  The degrees  of  the generators under this grading are defined  as follows:
 $$
 {\sf deg}({\sf 1}_\emptyset)=0
 \quad
  {\sf deg}({\sf 1}_\al)=0
  \quad
  {\sf deg} ({\sf spot}^\emptyset_\al)=1
    \quad
  {\sf deg} ({\sf fork}^\al_{\al\al})=-1
    \quad
  {\sf deg} ({\sf braid}^{\al\bet}_{\bet\al} )=0
 $$
for $\al,\bet \in S $   such that $|\al-\bet|>1$.
 \end{rmk}

\subsection{Khovanov  arc algebras}\label{Khovanov  arc algebras}

 We now recall the definition of  the extended Khovanov arc algebras studied in  \cite{MR2600694,MR2781018,MR2955190,MR2881300}. 
 We define $\mathcal{K} _{m,n }$ to be the algebra spanned by diagrams 
$$
\{
\underline{\la}
\mu \overline{\nu}
\mid \la,\mu,\nu \in \mptn \text{ such that }
\mu\overline{\nu},  \underline{\la}
\mu \text{  are oriented}\}
$$
 with the multiplication defined as follows.
First set 
 $$(\underline{\la}
\mu \overline{\nu})(\underline{\alpha}
\beta \overline{\gamma}) = 0 \quad \mbox{unless $\nu = \alpha$}.$$ 
To compute $(\underline{\la}
\mu \overline{\nu})(\underline{\nu}
\beta \overline{\gamma})$ place $(\underline{\la}
\mu \overline{\nu})$ under $(\underline{\nu}
\beta \overline{\gamma})$ then follow the \lq surgery' procedure.
This surgery combines two circles into one or splits one circle into two  using the following rules for re-orientation (where we use the notation
$1=\text{anti-clockwise circle}$, $x=\text{clockwise circle}$, $y=\text{oriented strand}$).  We have the splitting rules 
$$1 \mapsto 1 \otimes x + x \otimes 1,
\quad
 x \mapsto x \otimes x,
 \quad
 y \mapsto x \otimes y. 
 $$ 
 and the merging rules 
\begin{align*}
1 \otimes 1 \mapsto 1, 
\quad
1 \otimes x \mapsto x,
\quad
 x \otimes 1 \mapsto x,
\quad
x \otimes x \mapsto 0,
\quad
1 \otimes y  \mapsto y,\quad
x \otimes y \mapsto 0,
\end{align*}
\begin{align*}
y \otimes y  \mapsto
\left\{
\begin{array}{ll}
y\otimes y&\text{if both strands are propagating,  
one   is
 }\\
&\text{$\up$-oriented and  the other   is $\down$-oriented;}\\
0&\text{otherwise.} \end{array}
\right.
\end{align*}


\begin{eg}\label{brace-surgery}
We have the following product of Khovanov diagrams

$$   \begin{minipage}{3.1cm}
 %
 \end{minipage}  $$ 
 where we highlight with arrows the pair of arcs on which we are about to perform surgery.  The first equality follows from   the merging rule for $1\otimes 1 \mapsto 1$ and the second equality follows from  the merging rule 
$ 1 \mapsto 1 \otimes x + x \otimes 1$. 
 \end{eg}

\begin{eg}\label{brace-surgery2}
We have the following product of Khovanov diagrams

$$   \begin{minipage}{3.1cm}
 %
 \end{minipage}  $$ 
 where we highlight with arrows the pair of arcs on which we are about to perform surgery. 
This is similar to \cref{brace-surgery}. 
 \end{eg}

     \section{ Dyck combinatorics}
\label{dyckgens}
We have defined the $p$-Kazhdan--Lusztig polynomials via
 counting of certain 
oriented Temperley--Lieb diagrams.  
 For the purposes of this paper, we require richer combinatorial objects which {\em refine} the Temperley--Lieb construction: these are provided by tilings by Dyck paths.  
  
 Let us start with a simple example to see how these Dyck paths come from the  oriented Temperley--Lieb diagrams. Consider the partitions $\mu = (5^3,4, 1)$ and $\la = (4^2,3,1^2)$. The oriented Temperley--Lieb diagram $\varnothing e_\mu \la$ is illustrated in Figure \ref{example-2}. We see that $\la$ is obtained from $\mu$ by swapping the labels of the vertices of one cup in $e_\mu$. The tiles of $\mu$ which intersect this cup form a Dyck path (see definition below) highlighted in pink. Moreover the partition $\la$ is obtained from the partition $\mu$ by removing the equivalent Dyck path shaded in grey (or, equivalently, by removing the pink Dyck path, and letting the tiles fall under gravity).
More generally, if $\la, \mu\in \mptn$ with $\underline{\mu}\la$ oriented of degree $k$, then we will see that the partition $\la$ is obtained from the partition $\mu$ by removing $k$ Dyck paths.  

In this section, we develop the combinatorics of Dyck paths needed to give a quadratic presentation for the Hecke category.

\begin{figure}[ht!]
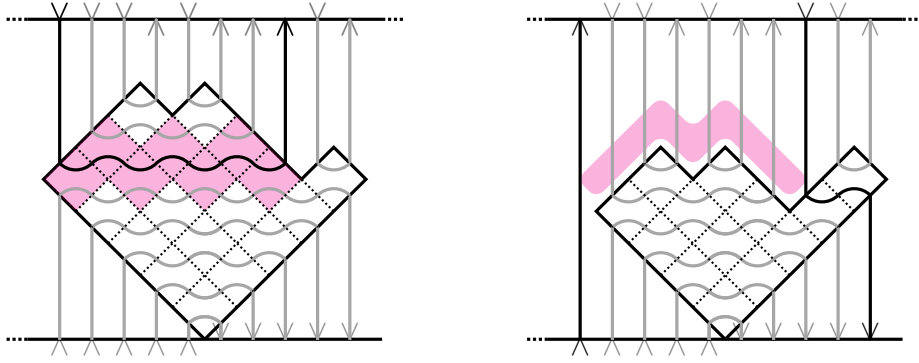

$$
 
   $$
   \vspace{-1cm}
   
   \caption{On the left we picture the cup diagram for $(5^3,4,1)$ and we highlight the arc, $p$, and the corresponding 
      path $P_{sf}$. On the right we have the
   partition/cup diagram obtained by removing $P$.  
   }
   \label{example-2}
  \end{figure}

    \subsection{Dyck paths}
 
 We define a path on the $m\times n$ tiled rectangle to    
be a finite non-empty set $P$
 of  tiles  that are ordered $[r_1, c_1], \ldots , [r_s, c_s]$ for some $s\geq 1$ such that 
 for each $1\leq i\leq s-1$ we have $[r_{i+1}, c_{i+1}] = [r_i+1, c_i]$ or $[r_i, c_i -1]$. Note that the set $\underline{\sf cont}(P)$ of contents of the tiles in a path $P$ form an interval of integers.
 We say that $P$ is a {\sf Dyck path} if
 $$\min \{ r_i+c_i-1 \, : \, 1\leq i\leq s\} = r_1+c_1-1 = r_s+c_s-1,$$
that is the minimal height of the path is achieved at the start and end of the path.
 We will write ${\sf first}(P) = {\sf cont}([r_1, c_1])$ and ${\sf last}(P) = {\sf cont}([r_s, c_s])$. 
 
 Throughout the paper, we will identify all Dyck paths $P$ having the same content interval ${\sf cont}(P)$. There are a few of places where we will need to fix a particular representative for a Dyck path $P$ and in that case we will use subscripts, such as $P_b$ or $P_{sf}$.

     Given $P$ a Dyck path, we set $|P|$ to be the number of tiles in $P$. We also define the {\sf breadth} of $P$, denoted by $b(P)$, to be  
 $$b(P)=
 \tfrac{1}{2}(|P|+1).$$   This measures the horizontal distance covered by the path.

\begin{defn} Let $P$ and $Q$ be Dyck paths.
\begin{itemize}[leftmargin=*]
\item We say that $P$ and $Q$ are {\sf adjacent} if and only if the multiset given by the disjoint union ${\sf cont}(P) \sqcup {\sf cont}(Q)$ is an interval.
\item We say that $P$ and $Q$ are {\sf distant} if and only if $$\min \{  |{\sf cont}[r,c] - {\sf cont}[x,y]| \, : \, [r,c]\in P, [x,y]\in Q\} \geq 2.$$
\item We say that $P$ {\sf covers} $Q$ and write $Q\prec P$ if and only if $${\sf first}(Q) >{\sf first}(P) \,\,  \mbox{and} \,\,  {\sf last}(Q) < {\sf last} (P).$$
\end{itemize}
 \end{defn}
 
 Examples of such Dyck paths $P$ and $Q$ are given in Figure \ref{adjdistcover}.

\begin{figure}[ht!]
$$   

 $$
\caption{Examples of $\color{magenta}P$ and $\color{cyan}Q$ adjacent, distant, and ${\color{cyan}Q }\prec \color{magenta}  P$ respectively.}
\label{adjdistcover}
\end{figure}

\subsection{Removable and addable Dyck paths}

Now we fix a partition $\mu\in \mptn$.  Recall that we identify any pair of Dyck paths which have  the same   content intervals.

\begin{defn}
Let $\mu \in \mptn$ and $P$ be a Dyck path. We say that $P$ is a {\sf removable  Dyck path} from $\mu$ if 
 there is a representative $P_{b}$ of $P$   such that $\la  :=  \mu\setminus P_b\in \mptn$.
 In this case we will write $\la = \mu-P$. (Note that this is well-defined as if $P_b$ exists then it is unique). 
 We define the set ${\rm DRem}(\mu)$ to be the set of all removable Dyck paths from $\mu$.
 
 We say that $P$ is an {\sf addable Dyck path} of $\mu$ if there is a representative $P_b$ of $P$ such that $\la := \mu \sqcup P_b \in \mptn$. In this case we will write $\la = \mu + P$. (Note again that this is well-defined as if $P_b$ exists then it is unique). We define the set ${\rm DAdd}(\mu)$ to be the set of all addable Dyck paths of $\mu$.  \end{defn}

 \begin{prop}\label{bijection}
Fix $\mu\in \mptn$.  There is a bijection between the set of cups in $e_\mu$ (or in $\underline{\mu}$) and the set ${\rm DRem}(\mu)$.
 \end{prop}
 \begin{proof}
 As observed at the beginning of this section, every cup in $e_\mu$ gives rise to a removable Dyck path. The fact that every removable Dyck path corresponds to a cup follows from the construction of the cup diagram $\underline{\mu}$.  
 \end{proof}
 
 \begin{defn}
 Let $\mu \in \mptn$. For each $P\in {\rm DRem}(\mu)$, we define $P_{sf}$ to be its representative given by the set of tiles intersecting the corresponding cup in $e_\mu$ as shaded in pink in Figure \ref{example-2}.
 \end{defn}

\begin{lem}\label{Remproperties}
Let $\mu\in \mptn$ and let $P,Q\in {\rm DRem}(\mu)$. Then either $P$ covers $Q$, or $Q$ covers $P$, or $P$ and $Q$ are distant.
\end{lem}

\begin{proof}
This follows directly from \cref{bijection}
\end{proof}

\begin{defn}
Let $\mu\in \mptn$ and $P,Q\in {\rm DRem}(\mu)$. We say that $P$ and $Q$ {\sf commute} if $P\in {\rm DRem}(\mu -Q)$ and $Q\in {\rm DRem}(\mu - P)$.
\end{defn}

\begin{lem}
Let $\mu\in \mptn$ and $P,Q\in {\rm DRem}(\mu)$.  
Then $P$ and $Q$ commute if and only if $P_{sf}\cap Q_{sf}=\emptyset$.
\end{lem} 
 
 \begin{proof}
 This follows directly by definition.
 \end{proof}
 
 \subsection{Oriented Temperley--Lieb diagrams and Dyck tiling}
 We now introduce  Dyck tilings and relate them to oriented arc diagrams.
 \begin{defn}\label{Dyckpair}
 Let $\lambda\subseteq  \mu\in \mptn$. A {\sf Dyck tiling} of the skew partition $\mu\setminus \la$ is a set $\{P^1,\dotsc,P^k\}$ of Dyck paths such that
 $$\mu\setminus \la = \bigsqcup_{i=1}^k P^i $$
 and for each $i\neq j$ we have either $P^i$ covers $P^j$ (or vice versa), or $P^i$ and $P^j$ are distant. 
 We call $(\la,\mu)$ a {\sf Dyck pair of degree $k$} if $\mu \setminus \la$ has a Dyck tiling with $k$ Dyck paths.
 \end{defn}

 We will see  that Dyck tilings are essentially unique, and as a consequence the degree of a Dyck pair is well defined. Examples of such tilings are given in \cref{cref-it} for the pair $(\la,\mu) = ((11,9,8,7,6,4,3^2,2^2), (11^7,8^3,2^2))$. We see that even though the tilings are different (as partitions of $\mu \setminus \la$), the Dyck paths appearing are the same (remember that we identify Dyck paths with the same content intervals).

\begin{figure}[ht!]
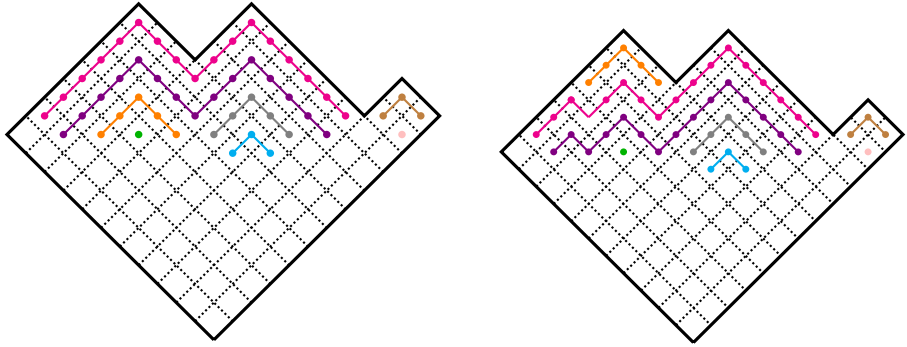
  
  $$
 
$$ 
  \caption{Two of the twelve   Dyck tilings of shape $
    (11^7,8^3,2^2)
    \setminus (11,9,8,7,6,4,3^2,2^2)$. 
  Compare with \cref{cref-it2}.
   }
  \label{cref-it}
  \end{figure}

 \begin{lem}\label{tilingrem}
 Let $\lambda \subseteq \mu\in \mptn$ with Dyck tiling $\mu\setminus \la = \sqcup_{i=1}^k P^i$ as in \cref{Dyckpair}. Then $P^i\in {\rm DRem}(\mu)$ for all $1\leq i \leq k$.
  \end{lem}

\begin{proof}
We prove it by induction on $k$. If $k=0$ there is nothing to prove. If $k=1$ then $\mu \setminus \la = P^1$ and so $P^1\in {\rm DRem}(\mu)$ as required. Now let $k\geq 2$ and assume that the result holds for $k-1$. Pick a removable tile $[x,y]$ in $\mu\setminus \la$ such that it belongs to some $P^j$ with $|P^j|$ minimal. Then we claim that $P^j\in {\rm DRem}(\mu)$. Indeed, if there were any tile $[r,c]$ above $P^j$ preventing it from being removable, then by the definition of Dyck pair and the minimality of $P^j$, we would have that $[r,c]$ belongs to a Dyck path $Q$ which covers $P^j$. But this would contradict the fact that $[x,y]$ is removable.
It remains to show that $P^i\in {\rm DRem}(\mu)$ for all $i\neq j$. Now, we have that $(\mu-P^j, \la)$ is a Dyck pair and so by induction $P^i\in {\rm DRem}(\mu-P^j)$ for all $i\neq j$. 
Fix $i\neq j$. If $P^i$ and $P^j$ are distant, then we have $P^i\in {\rm DRem}(\mu)$ as required.  If $P^i$ covers $P^j$ then we must have $|P^i|\geq |P^j|+4$, as it is impossible to have a partition $\mu$ and Dyck paths $Q, Q'$ with $|Q'|=|Q|+2$, $Q\in {\rm DRem}(\mu)$ and $Q'\in {\rm DRem}(\mu-Q)$. This means that we can shift the tiles of $P^j\in {\rm DRem}(\mu-P^i)$ with the same contents as those of $P^i\in {\rm DRem}(\mu)$ one step up and we get an equivalent Dyck path which is now removable from $\mu$ as required. This is illustrated in   \cref{YYY}.
Finally, suppose $P^j$ covers $P^i$. In this case we can again shift the tiles of $P^i\in {\rm DRem}(\mu - P^j)$ one step up so that it is now a subset of $P^j\in {\rm DRem}(\mu)$. We claim that this subset is also removable. If not, then we would have some $P^l$ which is adjacent to $P^i$, contradicting the fact that $(\la, \mu)$ is a Dyck pair. This case is illustrated in  \cref{ZZZ}.
\end{proof}

\begin{figure}[ht!]
$$\vspace{-1cm}   
 $$
  \vspace{-0.5cm}
\caption{Examples of the paths $\color{cyan}P^j$ and $\color{magenta}P^i$ such that the latter covers the former (and hence $|{\color{magenta}P^i}| \geq |{\color{cyan}P^j}| +4$).}
\label{YYY}
\end{figure}

\vspace{-0.5cm}
\begin{figure}[ht!]
$$  
  %
 
  $$
  
  \vspace{-1cm}
\caption{In both diagrams we depict  examples of  Dyck 
paths $\color{cyan}P^j$ and $\color{magenta}P^i$ such that the former  covers the latter.  On the left we see that  $\color{magenta}P^i$  is removable.  On the right we make $\color{magenta}P^i$ slightly larger so as to contradict removability:  in this case we must also include the 
Dyck path $\color{orange}P^l$ by our assumption that $\la$ is a partition;  however $\color{orange}P^l$ and  $\color{magenta}P^i$ are adjacent, a contradiction. 
}  \label{ZZZ}
\end{figure}

\begin{thm}\label{TLDyck}
Let $\la, \mu\in \mptn$. Then $\underline{\mu}\la$ is oriented if and only if $(\la,\mu)$ is a Dyck pair.
\end{thm}

\begin{proof}
Assume that $\underline{\mu} \la$ is oriented. Then the weight $\lambda$ is obtained from the weight $\mu$ by swapping the labels of pairs corresponding to some of the cups in $e_\mu$. Let $P^1, \ldots, P^k$ be the Dyck paths corresponding to these cups. We list these in order such that if $P^i\prec P^j$ then $j<i$. Then it's easy to see that $P^i\in {\rm DRem}(\mu-P^1 - \ldots - P^{i-1})$ for all $i$ and $\la = \mu - P^1 - \ldots - P^k$. 
It follows from \cref{Remproperties} that $\mu\setminus \la = \sqcup_{i=1}^k P^i$ is a Dyck tiling.
Conversely, suppose that $\la \subseteq \mu$ with $\mu\setminus \la = \sqcup_{i=1}^k P^i$ a Dyck tiling. Then it follows from \cref{tilingrem} that each $P^i\in {\rm DRem}(\mu)$ and so $\underline{\mu}\la$ is oriented.
\end{proof}

\begin{cor}\label{corDyckpairs} Let $\la\subseteq \mu$. Then $\mu\setminus \la = \sqcup_{i=1}^k P^i$ where the $P^i$'s are Dyck paths is a Dyck tiling if and only if $P^i\in {\rm DRem}(\mu)$ for all $i$. 
In particular, the set of Dyck paths $\{P^i \, : \, 1\leq i\leq k\}$ is unique. Moreover, in this case we have that ${\rm deg}(\underline{\mu}\la) = k$.
\end{cor}

  \subsection{ Dyck paths generated by  tiles.}
We will need one last piece of combinatorics to describe our quadratic presentation for the Hecke category.
     
\begin{defn}
Fix $\mu \in \mathscr{P}_{m,n}$ and $[r,c]\in \mu$. 
Let $l, k$ be the maximal non-negative integer such that $[r-i, c+i] \in \mu$ for all $0\leq i\leq l$, $[r-i+1, c+i]\in \mu$ for all $1\leq i\leq l$ and $[r+j, c-j]\in \mu$ for all $0\leq j\leq k$, $[r+j, c-j+1]\in \mu$ for all $1\leq j\leq k$. Then define the {\sf Dyck paths generated by the tile $[r,c]\in \mu$}, denoted by $\langle r,c \rangle_\mu$, to be the path 
$$[r-l, c+l], [r-l+1, c+l] ,\ldots , [r,c] , \ldots , [r+k, c-k].$$
\end{defn}  

Note that the Dyck path generated by a tile of a partition $\mu$ may or may not be in ${\rm DRem}(\mu)$ as illustrated in Figure \ref{genDyck}.

\begin{figure}[ht!]
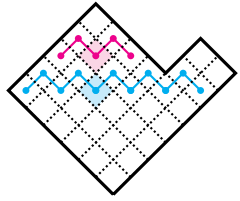


$$
 
$$

\caption{The path $\langle r,c\rangle _\mu$ for $[r,c]  =\color{magenta} [4,5]$
and 
$\color{cyan} [3,4]$ with $\mu=(6^5,2^2)$. 
We note that ${\color{magenta} \langle 4,5\rangle_\mu } \in {\rm DRem}	(\mu)
$, however 
${\color{cyan} \langle 3,4\rangle_\mu } \notin {\rm DRem}	(\mu)
$.
}
\label{genDyck}
\end{figure}

\section{ Generators   for the Hecke category }
\label{gens}

In this section, we lift the combinatorics of 
\cref{dyckgens} to provide a new set of generators for the Hecke category $\mathcal{H}_{m,n}$.  These generators are not `monoidal', 
 but all lie in degree $0$ or $1$ (this is the first step in providing a quadratic presentation). 
 
 \smallskip
 
 \subsection{Soergel diagrams from  oriented Temperley--Lieb diagrams }
We now revisit the classical definition of the light leaves basis starting from oriented Temperley--Lieb diagrams.  This material is covered in detail (from a slightly different perspective) in \cite{compan}.  
 \begin{defn}
 We   define up and down operators on diagrams as follows.
 Let $D$ be any Soergel graph with northern colour sequence $\sts\in {\rm Std}(\alpha)$ for some $\alpha\in \mptn$.
  \begin{itemize}[leftmargin=*]
\item    
Suppose that
$\csigma \in {\rm Add}(\alpha)$.  We define  $$
\qquad {\sf U}^1_\csigma(D)=\begin{minipage}{1.85cm}\begin{tikzpicture}[scale=1.5]
\draw[densely dotted, rounded corners] (-0.5,0) rectangle (0.75,0.75) node [midway] {$ D$} ;
 \end{tikzpicture}\end{minipage}
 \begin{minipage}{0.75cm}\begin{tikzpicture}[scale=1.5]
\draw[densely dotted, rounded corners] (-0.25,0) rectangle (0.25,0.75);
\draw[magenta,line width=0.08cm](0,0)--++(90:0.75);
 \end{tikzpicture}\end{minipage}
 \qquad 
 \qquad 
 {\sf U}^0_\csigma(D)=
\begin{minipage}{1.85cm}\begin{tikzpicture}[scale=1.5]
\draw[densely dotted, rounded corners] (-0.5,0) rectangle (0.75,0.75) node [midway] {$D$} ;
 \end{tikzpicture}\end{minipage}
\begin{minipage}{0.75cm}\begin{tikzpicture}[scale=1.5]
\draw[densely dotted, rounded corners] (-0.25,0) rectangle (0.25,0.75);
\draw[magenta,line width=0.08cm](0,0)--++(90:0.3525) coordinate (hi);
\draw[fill=magenta,magenta] (hi) circle (3pt);
 \end{tikzpicture}\end{minipage}
$$
 
\item   Now suppose that  
${\color{magenta}[r,c]} \in {\rm Rem}(\alpha)$   for $\alpha \vdash a$
and with ${\color{magenta}r-c}=\csigma \in S$.  
We let  $\stt \in  \Std (\alpha-{\color{magenta}[r,c]})$ be defined as follows: 
if $\sts{\color{magenta}[r,c]}=k$ then we let  
 $\stt^{-1}(j)=\sts^{-1}(j)$ for $ 1\leq j <k$ and 
  $\stt^{-1}(j-1)=\sts^{-1}(j)$ for $ k<j \leq a $. 
We  let $\stt \otimes \csigma  \in \Std(\alpha)$ be defined by $
(\stt \otimes \csigma){\color{magenta}[r,c]}=a$ and $(\stt\otimes \csigma)[x,y] =\stt[x,y]$ otherwise. 
We define 
$$
 {\sf D}_\csigma^0(D)= 
\begin{minipage}{1.85cm}
\end{minipage}.\qquad$$

\end{itemize}
\end{defn}

 \begin{defn} Let $\la, \mu \in \mptn$ with $(\la,\mu)$ a Dyck pair. Recall the oriented tableau $\stt_\mu^\la$ from Definition 2.2. Suppose that $\stt_\mu^\la[r_k,c_k] = (k,x_k)$ for each $1\leq k\leq \ell(\mu)$. We construct the Soergel graph $D_\mu^\la$ inductively as follows. 
 We set $D_0$ to be the empty diagram.  
Now let $1\leq k\leq \ell(\mu)$ with $\csigma = r_k-c_k$.
We set 
  $$D_k= \left\{ \begin{array}{ll} {\sf U}^1_{\csigma }(D_{k-1}) & \mbox{if $x_k = 1$;}\\[6pt] 
  
  {\sf U}^0_{\csigma }(D_{k-1}) & \mbox{if $x_k = s$;}\\[6pt] 
 {\sf D}^0_{\csigma }(D_{k-1}) & \mbox{if $x_k = f$;}\\[6pt] 
  {\sf D}^1_{\csigma }(D_{k-1}) & \mbox{if $x_k = sf$}. \end{array}\right.
  $$
 Now $D_{\ell(\mu)}$ has southern colour sequence $\stt_\mu$ and northern colour sequence some $\sts\in {\rm Std}(\la)$. We then define $D_\mu^\la = {\sf braid}_\sts^{\stt_\la} \circ D_{\ell(\mu)}$. 
 We also define $D^\mu_\la = (D^\la_\mu)^\ast$.

\end{defn}
 
%

\begin{eg}
 In \cref{twopics} we provide an example of  the labelling of an
  oriented Temperley--Lieb diagram and the corresponding up-down construction of a basis element.

\end{eg}

\begin{figure}[ht!]
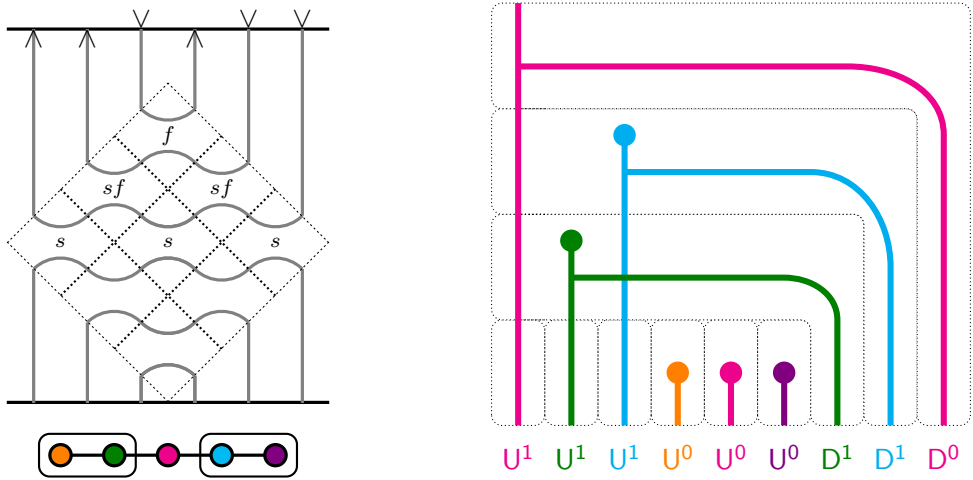

$$\qquad
 \begin{minipage}{4cm}
\end{minipage}
$$
\caption{ On the left we depict the   labelling of the   oriented  Temperley--Lieb diagram of shape $\la=(1)$ and $\mu=(3^3)$.
On the right we depict the unique  $D^\la_\mu$ for  $\stt_{(3^3)}\in \Std((3^3))$.  }
\label{twopics}
\end{figure}

 \begin{thm}\label{heere ris the basus}
 The  algebra   
$ \mathcal{H}_{m,n}  $  is a graded cellular (in fact quasi-hereditary) algebra with  graded cellular basis given by
\begin{equation}\label{basis}
\{D_\la^\mu D^\la_\nu \mid  \la,\mu,\nu\in \mptn \,\, \mbox{with}\,\,		(\la,\mu), (\la,\nu) \text{ Dyck pairs}  \}
\end{equation}
with 
$${\rm deg}( D_\la^\mu D^\la_\nu) = {\rm deg}(\la, \mu) + {\rm deg}(\la, \nu),$$
with respect to the involution $*$ and the partial order on $\mptn$ given by inclusion.
\end{thm}
\begin{proof}
This is simply a combinatorial rephrasing of the light leaves basis, constructed in full generality in  \cite{MR3555156,MR2441994} and 
reproven in the case of  $(W,P)= (S_{m+n}, S_m \times S_n)$ paper in \cite{MR4323501}.
\end{proof}

Note, in particular that the degree $0$ basis elements are given by $D_\la^\la = {\sf 1}_\la$, $\la\in \mptn$ and the degree $1$ basis elements are given by $D_\mu^\la$ and $D_\la^\mu$ for $\la, \mu \in \mptn$ with $\la = \mu - P$ for some $P\in {\rm DRem}(\mu)$.  We will show that these degree 0 and degree 1 elements generate $\mathcal{H}_{m,n}$.
But first we will describe an easy way of visualising products of light leaves basis elements directly from the oriented tableaux used to define them.

\subsection{Multiplying generators on the oriented tableaux}
We have seen that the Soergel graph $D_\mu^\la$ is completely determined by the oriented tableau $\stt_\mu^\la$. To visualise the multiplication of two such Soergel graphs directly from the oriented tableaux, we will want to consider pairs of  tableaux of the same shape. This can easily be done by adding one more possible orientation for tiles, namely $0$. When constructing the corresponding Soergel graph, whenever we encounter a tile with $0$-orientation we will simply tensor with the empty Soergel graph, that is we leave the graph unchanged (see \cref{blow-up} for two examples).

\begin{figure}[ht!]
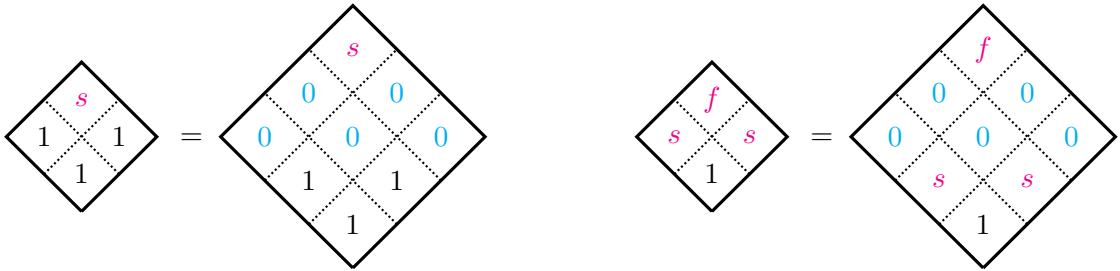


$$
\begin{minipage}{2.2cm}
 \end{minipage}
$$

\caption{Re-drawing tableau of shape $(2^2)$ as tableaux of shape $(3^3)$ where $(2^2)=(3^3)-\color{cyan}P$ and $\color{cyan}P_{sf}$ is depicted as the blue zeroes.}
\label{blow-up}
\end{figure}

For example, let  $P \in {\rm DRem}(\mu)$ and $Q\in {\rm DRem}(\mu-P)$ and assume for now that $P$ and $Q$ commute. We would like to be able to visualise the product $$D^{\mu - P -Q}_{\mu - P} D^{\mu - P}_\mu$$ on the oriented tableaux. So instead of considering the oriented tableau $\stt_{\mu - P}^{\mu - P - Q}$ as a labelling of the tiles of $\mu - P$, we can visualise it as a labelling of the tiles of $\mu$ with all tiles belonging to $P_{sf}$ having $0$-orientation. 
The orientation of the other tiles remain unchanged.  
We can then easily multiply the elements  $D^{\mu - P -Q}_{\mu - P}$ with $D^{\mu - P}_\mu$ simply by  \lq stacking' the two oriented tableaux, without any need to apply a braid generator in between the two diagrams. 
An example is depicted in \cref{zeroorientation}.

\begin{figure}[ht!]
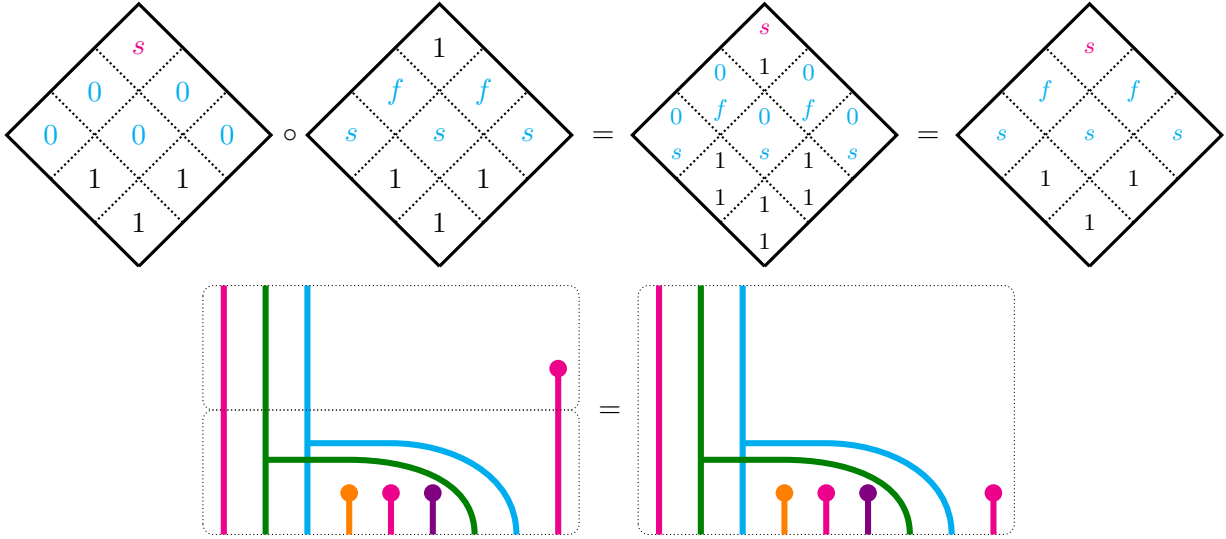


%
\end{minipage}
$$
\caption{A product of commuting diagrams on tableaux.}
\label{zeroorientation}
\end{figure}

Now,  let  $P \in {\rm DRem}(\mu)$ and $Q\in {\rm DRem}(\mu-P)$ be such
 that $P$ and $Q$ do not commute. 
We proceed as above (rewriting the tiles in $P_{sf}$ so as to have a $0$-orientation)  and then we let each tile with an $s$-orientation fall down one place (from $[r,c]$ to $[r-1,c-1]$, say) and we leave all other tiles unchanged.  
An example is given in \cref{zeroorientation2}. 
When multiplying two elements in this way, we will represent the product by splitting each tile in half with the label of the top half corresponding to the first element and the label of the bottom half correspond to the second element.

When considering the dual Soergel graphs $D_\la^\mu = (D_\mu^\la)^*$, we will represent them with the same oriented tableau as $D_\mu^\la$ except that we will replace all $s$-orientation (respectively $f$-orientation, or $sf$-orientation) by the symbol $s^*$ (respectively $f^*$, or $f^*s^*$). An example of a degree two basis element $D^\mu_\la D^\la_\nu$ is given in Figure \ref{zeroorientation2}.
We will now restate some of the (simplest) relations in the Hecke category on the oriented tableaux.

\begin{prop}
The relations depicted in \cref{spotpic,spotpic2} hold.

\begin{figure}[ht!]
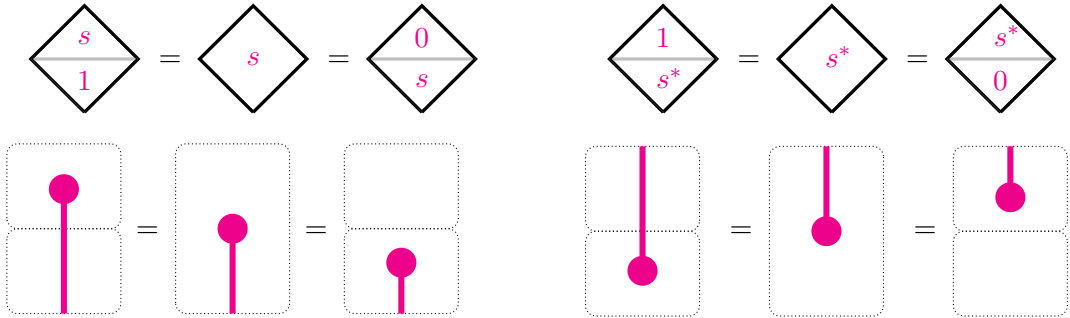


$$
\begin{minipage}{1.6cm}
\end{minipage} 
$$
\caption{The (dual) spot idempotent relations on tableaux.}
\label{spotpic}
\end{figure}

\begin{figure}[ht!]
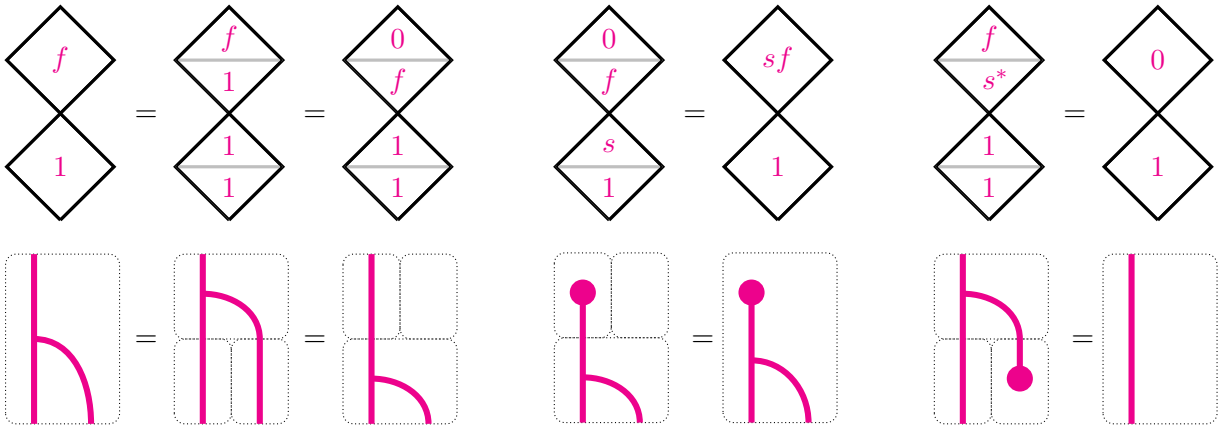

$$
\begin{minipage}{1.6cm}
%
\end{minipage} 
  $$

\caption{A few fork (and fork-spot) relations on tableaux.  }
\label{spotpic2}
\end{figure}

\end{prop}

\begin{proof}
These are all restatements of the relations given in the monoidal presentation of the Hecke category. For example the last one is the dual fork-spot contraction.
\end{proof}

 \subsection{Generators for the Hecke category}

     We are now ready to prove that the algebra $\mathcal{H}_{m,n}$ is generated in degree 0 and 1.
\begin{prop}\label{generatorsarewhatweasay}
The algebra $\mathcal{H}_{m,n}$ is generated by the elements 
$$\{D^\la_\mu,
D_\la^\mu \mid 	
\la, \mu \in \mptn \text{ with $\la = \mu - P$ for some $P\in {\rm DRem}(\mu)$} 	\}\cup\{ D_\mu^\mu = {\sf 1}_\mu \mid \mu \in \mptn		\}.
$$
\end{prop}
 \begin{proof}
 lt suffices to show that every element $D_\mu^\la$  for $\la, \mu\in \mptn$ with $(\la,\mu)$ a Dyck pair can be written as a product of these elements. 
 We proceed by induction on $k = {\rm deg}(\underline{\mu}\la)$. For $k=0$ or $1$, there is nothing to prove. So assume that $k\geq 2$. We have $\mu\setminus \lambda = \sqcup_{i=1}^k P^i$ where each $P^i\in {\rm DRem}(\mu)$. Pick $P\in \{P^i \, : \, 1\leq i\leq k\}$ such that there is no $P^i$ covering $P$. Then we claim that 
 $$D_\mu^\la = D_{\mu - P}^\la D_{\mu}^{\mu - P}.$$
 The result would then follow by induction. To see this, note that the oriented tableau $\stt_{\mu -P}^\la$, viewed as a tableau of shape $\mu$ as explained in the last subsection is obtained from $\stt_\mu^\la$ by setting the orientation of all tiles of $P_{sf}$ to $0$. Moreover, if there is some $P^i\neq P$ which does not commute with $P$, then each of the $s$-orientations on the tiles of $P^i_{sf}$ (which also belong to $P_{sf}$) in $\stt_\mu^\la$ fall down one tile. Note that by assumption on $P$, these were labelled by $1$ in $\stt_\mu^\la$.  
 Now using the relations given in \cref{spotpic,spotpic2}, we see that $D_{\mu - P}^\la D_{\mu}^{\mu - P} = D_\mu^\la$ as required. The dual element $D_\la^\mu = (D^\la_\mu)^*$ can then by written as the reverse product of the dual degree 1 elements. An example is given in Figure \ref{zeroorientation2}.
 \end{proof}

\vspace{-0.1cm}
\begin{figure}[ht!]
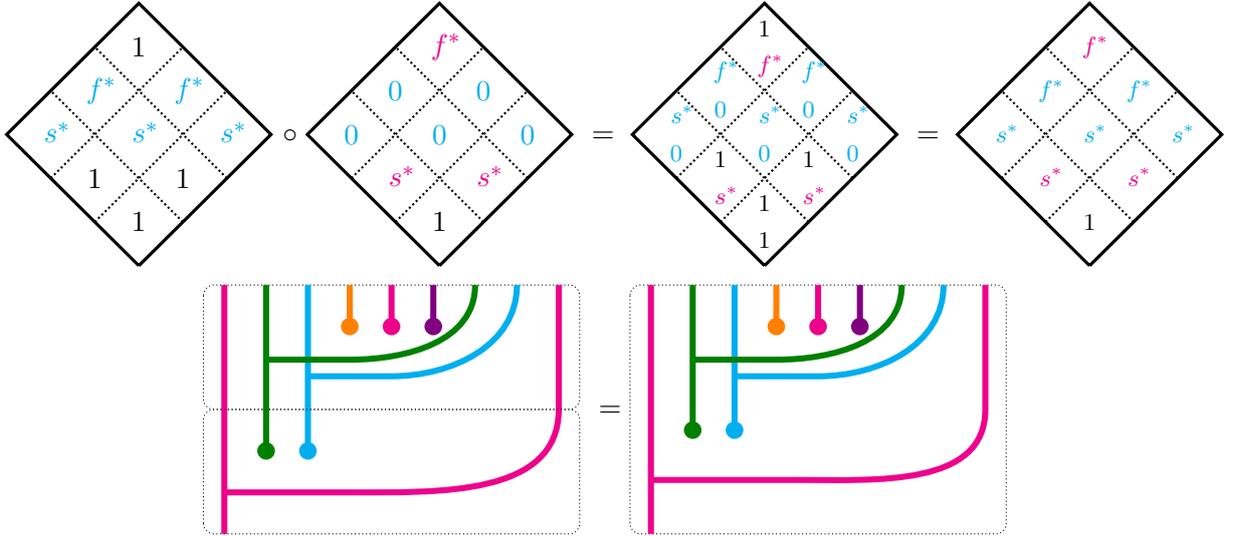


%
\end{minipage}
$$

\caption{A product of non-commuting diagrams on tableaux.}
\label{zeroorientation2}
\end{figure}

\vspace{-0.1cm}
 
\section{Dilation and contraction  }


We now make a slight detour in order to construct  dilation  maps which allow us to interrelate    partitions,  weights,   Dyck paths as well as Hecke categories and Khovanov arc algebras of different sizes.

\begin{defn}
For  $-m \leq k \leq n$ we define the dilation map $\varphi_k: \mathscr{P}_{m,n} \longrightarrow \mathscr{P}_{m+1,n+1}$ on weights by setting $\varphi_k (\la)$ for $\la \in \mathscr{P}_{m,n}$ to be the weight obtained from $\la$ by moving any label in position $x<k$ to $x-1$, any label in position $x>k$ to $x+1$ and labelling the vertices $k-\frac{1}{2}$ and $k+\frac{1}{2}$ by $\down$ and $\up$ respectively. 
\end{defn}

The following lemmas follow directly from the definition.

\begin{lem}
The map $\varphi_k$ is injective with image given by the set $\mathscr{P}^k_{m+1,n+1}$ consisting of all partitions with a removable node of content $k$. We call such partitions {\sf contractible at $k$}.
\end{lem}

\begin{lem} Let $\la, \mu \in \mptn$ and let $-m \leq k \leq n$.
We have that $(\la,\mu)$ is a Dyck pair of degree $j$ if and only of $(\varphi_k(\la),\varphi_k(\mu))$  is a Dyck pair of degree $j$.
In particular, if $\la = \mu - P$ for some $P\in {\rm DRem}(\mu)$ then we have $\varphi_k(\la) = \varphi_k(\mu) - Q$ where $Q\in {\rm DRem}(\varphi_k(\mu))$ satisfies 
 
$$

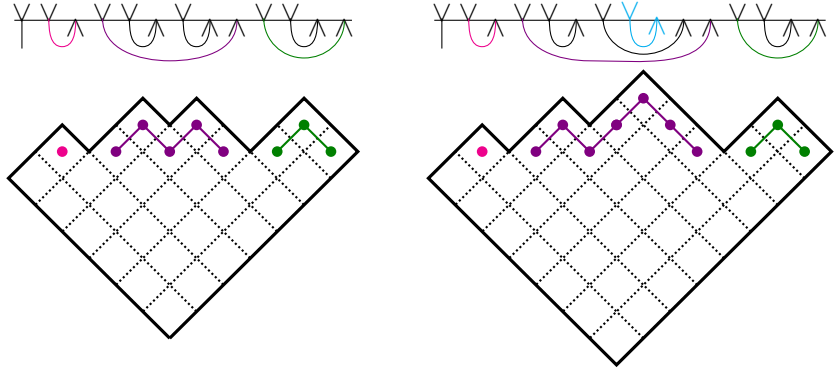
\caption{ The latter partition is obtained from the former by dilation at $\color{cyan}k=1$.   }

\end{figure}

 We now extend the dilation map $\varphi_k$ to dilation homomorphisms for the Hecke categories and the arc algebras. We will use the same notation for all three dilation maps. 
 We start with the Hecke category.  

\begin{thm}
Let $\Bbbk$ be a commutative integral domain and let $i \in \Bbbk$ be a square root of $-1$.  
 For $-m\leq k\leq n$, we define the map $\varphi_k : \mathcal{H}_{m,n }\to \mathcal{H}_{ m+1,n+1 }$ on the generators as follows. For $\la, \mu \in \mptn$ with $\la = \mu - P$ for some $P\in {\rm DRem}(\mu)$ we have
 $
\varphi_k({\sf 1}_\mu)= {\sf 1}_{{\varphi_k(\mu)}},
 $
and 
$$
\varphi_k(D^\la_\mu)= \left\{ \begin{array}{ll} D^{\varphi_k(\la)}_{\varphi_k(\mu)} & \mbox{if $k\notin \underline{\sf cont}(P)$} \\
(-i) \cdot 
D^{\varphi_k(\la)}_{\varphi_k(\mu)}	& \mbox{ if $k\in \underline{\sf cont}(P)$ and $k$ labels a spot tile in $P_{sf}$}\\	
i \cdot 
D^{\varphi_k(\la)}_{\varphi_k(\mu)}	& \mbox{ if $k\in \underline{\sf cont}(P)$ and $k$ labels a fork tile in $P_{sf}$} 
\end{array} \right.
$$ and $\varphi_k(D_\la^\mu) = \varphi_k((D_\mu^\la)^*) = (\varphi_k(D_\mu^\la))^*$.
Then $\varphi_k$ extends to an injective   homomorphism of graded $\Bbbk$-algebras. 
\end{thm}

\begin{proof}
The map $\varphi_k$ is defined on the monoidal  (spot, fork, braid and idempotent) generators of $\mathcal{H}_{m,n}$  in \cite[Section 5.3]{compan} (note that in that paper we use the notation $\varphi_\ctau$ where $\ctau = s_k$), where it is proven to be an injective   homomorphism of graded $\Bbbk$-algebras.  
Rewriting this in terms of the generators  $D^\la_\mu$  and 
$D^{\varphi_k(\la)}_{\varphi_k(\mu)}$ 
we deduce the result. 
 \end{proof}

 We now define the dilation homomorphisms for $\mathcal{K}_{m,n }$.

\begin{thm}\label{dilate-K}
Let $\Bbbk$ be a commutative integral domain. 
  For $-m\leq k\leq n$, the map $\varphi_k : \mathcal{K} _{m,n }\to \mathcal{K}_{m+1,n+1 } $ defined on arc diagrams by 
$$
\varphi_k( {\underline{\mu}\la \overline{\nu}} )= \underline{\varphi_k(\mu)} \varphi_k(\la) \overline{\varphi_k(\nu)}
$$ 
extends to an injective   homomorphism of graded $\Bbbk$-algebras.
\end{thm}
\begin{proof}
Consider a pair of arc diagrams for which 
the  $k\pm\tfrac{1}{2}$ vertices form anti-clockwise oriented circles.  
We can choose to do the surgery procedure so that this is the final step we consider.
The anti-clockwise circle acts as an idempotent and so the result follows.
  \end{proof}

 These dilation homomorphisms will allow us to prove results by induction. 
The base cases for the induction will be the following:
 
 \begin{defn}
Let $\la, \mu \in \mptn$ with  $\la=\mu-P$ for $P\in {\rm DRem}(\mu)$.  
 We say that $(\la,\mu)$ are {\sf incontractible} if there does not exist $k\in \ZZ$ such that $\la, \mu\in \mathscr{P}^k_{m,n}$.  
 \end{defn}
 
 \begin{rmk}\label{incontr-prtns}
By definition, it is clear that $(\la,\mu)$ are incontractible if and only if $\mu=(c^r)$ is a rectangular partition and 
$\la=(c^{r-1},c-1)$ so that $b(P)=1$.
\end{rmk}

    \section{The  quiver and relations for    $\mathcal{H}_{m,n }$}
 
We now provide a  presentation for $\mathcal{H}_{m,n }$ 
 over   $\Bbbk$  an arbitrary  integral domain. 
Before stating the presentation as a theorem, we first recall (and slightly rephrase) a proposition from \cite[Proposition 4.18]{compan} and a lemma which will be used in the proof.

\begin{prop}\cite[Proposition 4.18]{compan} \label{generationofspots}
Let $\la \in \mptn$,  
     ${\color{cyan}[r,c]} \in {\rm Add}(\la)$ such that $s_{r-c} = \ctau$.
     Then we have that 
\begin{align}\label{notChook2} 
{\sf 1}_{\la}  \otimes {\sf bar}(\ctau  )  
  = 
\sum_{[x,y]}
(-1) ^{b \langle x,y  \rangle_{\la} } 
D^{\la}_{\la- \langle x,y \rangle_{\la} } D^{\la - \langle x,y \rangle_{\la}  }_{\la}.
  \end{align}
where the sum is taken over all $[x,y]\in \la$ where either $[x,y]=[x,c]$ with $x<r$ or $[x,y]=[r,y]$ with $y<c$, and $\langle x,y \rangle_{\la}\in {\rm DRem}(\la)$.
\end{prop}

\begin{proof}
Given $\la \in \mptn$,  
     $[x,y]\in \la$ such that $\stt_\la([x,y])=k$ and $\csigma = s_{x-y}$,   we set 
   $$
  {\sf gap}(\stt_\la-[x,y]) = {\sf 1}_{\stt_\la{\downarrow}_{\{1,\dots,k-1\}}} \otimes {\sf spot}^{\csigma}_\emptyset
   {\sf spot}_{\csigma}^\emptyset
   \otimes {\sf 1}_{\stt_\la{\downarrow}_{\{k+1,\dots,\ell(\la)\}}}.  
  $$
It was shown in \cite[Proposition 4.18]{compan} that 
\begin{align}\label{notChook} 
{\sf 1}_{  \la}  \otimes {\sf bar}(\ctau  )  
  = 
-\!\!
\sum_{[x,y]}
  {\sf gap}(\stt_\la- [x,y])  
  \end{align}
  where the sum is taken over all $[x,y]\in \la$ where either $[x,y]=[x,c]$ with $x<r$ or $[x,y]=[r,y]$ with $y<c$. 
For each term in the sum, we now apply $b\langle x,y \rangle_\la -1$ times the null-braid relation to get 
$${\sf gap}(\stt_\la- [x,y])  = (-1)^{b\langle x,y \rangle_\la -1} D^{\la}_{\la- \langle x,y \rangle_{\la} } D^{\la - \langle x,y \rangle_{\la}  }_{\la}.$$
 If $\langle  x, y \rangle_\la \in {\rm DRem}(\la)$, then this is a basis element and we are done. If $\langle  x, y \rangle_\la \notin {\rm DRem}(\la)$, write $\langle  x, y \rangle_\la  = [x_1, y_1], \ldots , [x_s, y_s]$. Then we have either $[x_1, y_1+1]\in \la$ and $[x_1 -1, y_1+1]\notin \la$, or $[x_s +1, y_s]\in \la$ and $[x_s +1, y_s -1]\notin \la$. In both cases, we can apply the cyclotomic relations to deduce that the correspondent element in the sum is zero.
\end{proof}

\begin{figure}[ht!]
\vspace{-0.5cm}$$
 $$

\caption{On the left we picture the 
boxes involved in relation in \cref{notChook2}  for $\mu=(7^5)$ and $b(P)=1$.  
On the right we depict the Dyck paths generated by these boxes --- this is how we will rephrase
the dilated version of \cref{notChook2}  as the ``self-dual relation" in \cref{presentation} below. 
The Dyck paths $\color{gray}Q_{-1}$, 
$\color{brown}Q_{1}$
are adjacent to $P$ (where $P$ is pictured as a dot) 
$\color{darkgreen}Q_{2}$, $\color{orange}Q_{3}$, $\color{violet}Q_{4}$ 
highlighted so their colours correspond to the boxes in the left picture.    
}
\label{sum-relation}
\end{figure}

\begin{lem}\label{nullbraidontiles}
The following relations hold.
$$ 
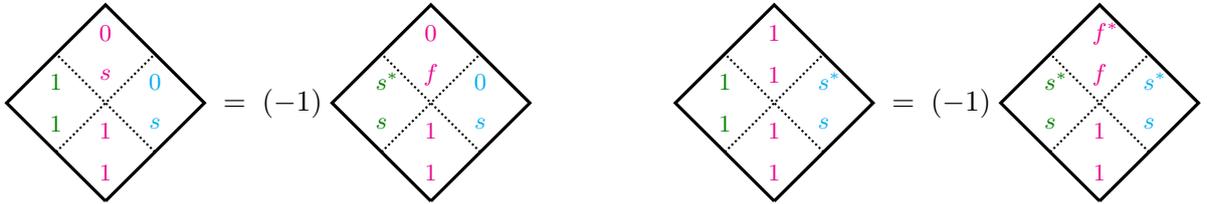
\begin{minipage}{2.55cm}
  \end{minipage}
 $$

\end{lem}

\begin{proof}
%
%
The first equation follows by applying the null braid relation followed by the fork-spot contraction as follows: 
$$
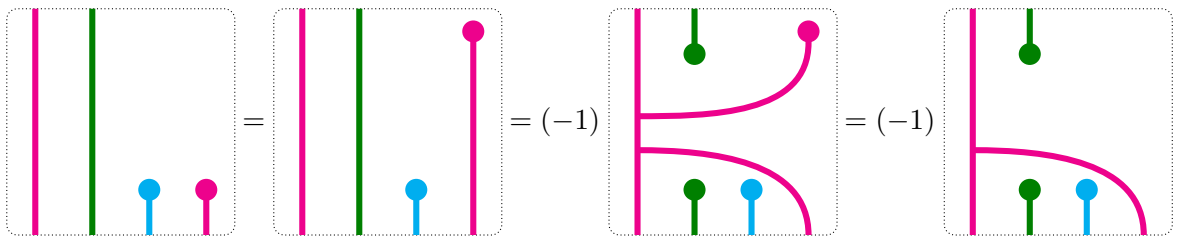
\begin{minipage}{3cm}%
\end{minipage} 
$$
(where the first equality is merely a trivial isotopy).  
The second equation  is obtained by simply applying the null braid relation as follows: 
 $$  
\begin{minipage}{3cm}%
\end{minipage}  
$$
as required. \end{proof}

\begin{thm}\label{presentation}
The algebra $\mathcal{H}_{m,n }$ is the  associative $\Bbbk$-algebra generated by the elements 
\begin{equation}\label{geners}
\{D^\la_\mu,
D_\la^\mu \mid 	
\text{$\la, \mu\in \mptn$ with $\la = \mu - P$ for some $P\in {\rm DRem}(\mu)$} 
	\}\cup\{ {\sf 1}_\mu \mid \mu \in \mptn \}		
	\end{equation}
	subject to the  following relations and their duals. 

	\smallskip\noindent
{\bf The idempotent   
relations:} 
For all $\la,\mu \in \mptn$, we have that 
\begin{equation}\label{rel1}
{\sf 1}_\mu{\sf 1}_\la =\delta_{\la,\mu}{\sf 1}_\la \qquad 
\qquad {\sf 1}_\la D^\la_\mu {\sf 1}_\mu = D^\la_\mu.
\end{equation} 

\smallskip\noindent
	{\bf The 
	self-dual relation: } 
	Let  $P\in {\rm DRem}(\mu)$ and $\la = \mu - P$. Then we   have 
\begin{equation}\label{selfdualrel}
D_\mu^{\la} D_{\la}^\mu
= (-1)^{b(P)-1}\Bigg(
2
\!\! \sum_{   \begin{subarray}{c} Q\in {\rm DRem}(\la) \\ P\prec Q \end{subarray}}
\!\!
(-1) ^{b(Q) } D^{\la}_{\la- Q } D^{\la - Q  }_{\la} + 
\!\!\sum_{  \begin{subarray}{c} Q\in {\rm DRem}(\la) \\ Q \, \text{adj.}\, P \end{subarray}}
\!\!
 (-1)^{b(Q)} D^{\la}_{\la- Q } D^{\la - Q  }_{\la}\Bigg) 
\end{equation}  
where we abbreviate ``adjacent to" simply as ``adj."  

\smallskip\noindent
{\bf The commuting relations:} 
Let $P,Q\in {\rm DRem}(\mu)$ which commute. Then we have 
\begin{equation}\label{commuting}
D^{\mu -P-Q}_{\mu-P}D^{\mu-P}_\mu = D^{\mu-P-Q}_{\mu -Q}D^{\mu -Q}_\mu \qquad
D^{\mu -P}_\mu D^\mu_{\mu - Q} = D^{\mu - P}_{\mu - P - Q}D^{\mu - P - Q}_{\mu - Q}.
\end{equation}

\smallskip\noindent
{\bf The non-commuting relation:}   
Let $P,Q\in {\rm DRem}(\mu)$ with $P\prec Q$ which do not commute. Then  $Q\setminus P = Q^1\sqcup Q^2$ where $Q^1, Q^2 \in {\rm DRem}(\mu - P)$ and we have 
\begin{equation}\label{noncommuting}
D^{\mu -Q}_\mu D^\mu_{\mu-P} = D^{\mu - Q}_{\mu - P - Q^1}D^{\mu - P - Q^1}_{\mu - P} = D^{\mu - Q}_{\mu - P - Q^2}D^{\mu - P - Q^2}_{\mu - P}
\end{equation}

\smallskip\noindent
{\bf The adjacent relation:}  
Let $P\in {\rm DRem}(\mu)$ and $Q\in {\rm DRem}(\mu - P)$ be adjacent. Denote by $\langle P\cup Q \rangle_\mu$, if it exists,  the smallest removable Dyck path of $\mu$ containing $P\cup Q$. Then we have 
\begin{equation}\label{adjacent}
D^{\mu - P - Q}_{\mu - P}D^{\mu - P}_\mu = \left\{ \begin{array}{ll} (-1)^{b(\langle P\cup Q\rangle_\mu) - b(Q)} D^{\mu - P - Q}_{\mu - \langle P\cup Q\rangle_\mu}D^{\mu - \langle P\cup Q\rangle_\mu}_\mu & \mbox{if $\langle P\cup Q\rangle_\mu$ exists} \\ 0 & \mbox{otherwise} \end{array} \right.
\end{equation}
  \end{thm}

\begin{eg}
 We have already seen   examples of the 
``commuting relations" in \cref{zeroorientation}; 
of the ``non-commuting relations" in \cref{zeroorientation2}; the ``adjacent relation" in \cref{nullbraidontiles};
the 
 ``idempotent relations" are what one would expect;  the combinatorics of 
 the ``self-dual relation" is pictured in \cref{generationofspots,sum-relation}.

\end{eg}

\begin{proof}
By \cref{generatorsarewhatweasay} it is enough to check that   
(\ref{rel1}) to (\ref{adjacent})     are a complete list of relations. 
 We first prove that all these relations do hold.  
The idempotent relations are immediate.  We now proceed to check the other relations.

\smallskip
\noindent\textbf{Proof of the self-dual relation.} First consider the case where $b(P) = 1$. This is (up to commutation) exactly the case covered in \cref{generationofspots}. We just need to note that $\langle r-1, c\rangle_\la$ and $\langle r, c-1\rangle_\la$ give Dyck paths adjacent to $P$ and that $Q = \langle r-j, c\rangle_\la = \langle r, c-j\rangle_\la$ (for $j\geq 2$) give two identical Dyck paths satisfying $P\prec Q$.

Now if $b(P)\geq 2$ then we can find $k$ such that $(\la, \mu) = (\varphi_k(\la'), \varphi_k(\mu'))$ where $\la' = \mu'-P'$ with $b(P')=b(P)-1$ and $k\in \underline{\sf cont}(P')$. Now using induction and applying the dilation homomorphism we get
\begin{eqnarray*}
-D_\mu^{\la} D_{\la}^\mu
&=& (-1)^{b(P')-1}\left(
2
\!\! \sum_{   \begin{subarray}{c} Q'\in {\rm DRem}(\la') \\ P'\prec Q' \end{subarray}}
\!\!
(-1) ^{b(Q') }(-1) D^{\varphi_k(\la')}_{\varphi_k(\la'- Q') } D^{\varphi_k(\la' - Q')  }_{\varphi_k(\la')} \right. \\
&& + 
\left. \!\!\sum_{  \begin{subarray}{c} Q'\in {\rm DRem}(\la') \\ Q' \, \text{adjacent to}\, P' \end{subarray}}
\!\!
 (-1)^{b(Q')} D^{\varphi_k(\la')}_{\varphi_k(\la'- Q') } D^{\varphi_k(\la' - Q')  }_{\varphi_k(\la)}\right). 
\end{eqnarray*}
(Note that the $(-1)$ in the first summand comes from the fact that $k\in \underline{\sf cont}(Q')$ whereas this is not the case for the terms in the second summand). 
Now we have $\varphi_k(\la'-Q') = \la - Q$ for some $Q\in {\rm DRem}(\la)$. Moreover, $\varphi_k$ gives bijections between all $Q'$s satisfying $P'\prec Q'$ and all $Q$s satisfying $P\prec Q$, and between all $Q'$s adjacent to $P'$ and all $Q$s adjacent to $P$. Finally, noting that $b(P) = b(P')+1$, $b(Q) = b(Q')+1$ if $P'\prec Q'$ and $b(Q) = b(Q')$ if $Q'$ adjacent to $P$ gives the required relation.  

\medskip
\noindent 
\textbf{Proof of the commuting relations.}
The fact that $P$ and $Q$ commute is equivalent to $P_{sf}\cap Q_{sf} = \emptyset$. Now these relations follow directly by applying the spot idempotent and fork idempotent relations illustrated in \cref{spotpic,spotpic2}.  
(An example is depicted in \cref{zeroorientation}.)  

\smallskip
\noindent 
\textbf{Proof of the non-commuting relations.}
First consider the case where $b(P) = 1$ and $b(Q)=2$ then we have $Q\setminus P = Q^1 \sqcup Q^2$ with $|Q^1|=|Q^2| = 1$.   Visualising the product of basis elements on the partition $\mu$ we have that the only non-trivial part is given in \cref{ADD DIAGRAM}, simply by applying   the spot-fork relation (as in \cref{spotpic2}). This proves the result in this case.

\begin{figure}[ht!]
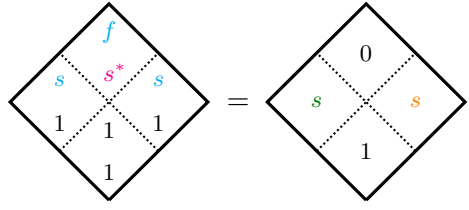

$$\begin{minipage}{2.55cm}
  \end{minipage}
$$
\caption{The non-commuting relations with $b({\color{magenta}P})=1$ and 
$b({\color{cyan}Q})=2$ on the left.  
On the right we picture ${\color{darkgreen}Q^1}$ and  ${\color{orange}Q^2}$. 
The equality follows by the spot-fork relation (as in \cref{spotpic2}).}
\label{ADD DIAGRAM}
\end{figure}
 
 Now if $b(P)\geq 2$ or $b(Q)\geq 3$ then we can find $k$, such that $(\mu, \mu -P, \mu -Q, \mu - P - Q^1, \mu - P - Q^2) = (\varphi_k(\mu'), \varphi_k(\mu' - P'), \varphi_k(\mu' - Q'), \varphi_k(\mu' -P' - Q'^1), \varphi_k(\mu' -P' - Q'^2))$ such that $k\in \underline{\sf cont}(P')\sqcup \underline{\sf cont}(Q'^1) \sqcup \underline{\sf cont}(Q'^2)$. Now we can use induction and apply the dilation homomorphism. Noting that spot (respectively fork) tiles in $P_{sf}$ correspond to fork (respectively spot) tiles in $Q_{sf}$, gives the required relation.
 (The example in \cref{zeroorientation2} is obtained by dilating the example in \cref{ADD DIAGRAM} at $k=0$.)

\smallskip
\noindent 
\textbf{Proof of the adjacent relation.}  First assume that $b(P)=b(Q)=1$. Say $Q={[r,c]}$ then we have $\langle P\cup Q \rangle_\mu = \langle r,c\rangle_\mu$. Now applying \cref{nullbraidontiles} (i) once followed by $b(\langle r,c\rangle_\mu)-2$ times \cref{nullbraidontiles} (ii), we obtain
$$D^{\mu -P-Q}_{\mu-P}D^{\mu - P}_\mu = (-1)^{b(\langle r,c\rangle_\mu)-1}D^{\mu-P-Q}_{\mu - \langle r,c\rangle_\mu}D^{\mu - \langle r,c\rangle_\mu}_\mu$$
Now if $\langle r,c\rangle_\mu\notin {\rm DRem}(\mu)$ then this is equal to zero by the cyclotomic relations, giving the results. An example of this case is provided in \cref{a big one}. 
Now if $b(P)\geq 2$ or $b(Q)\geq 2$ then we can find $k$ such that 
$(\mu, \mu -P , \mu -P-Q, \mu-\langle P \cup Q\rangle_\mu) = (\varphi_k(\mu'), \varphi_k(\mu' -P') , \varphi_k(\mu' -P'-Q'), \varphi_k(\mu'-\langle P' \cup Q'\rangle_\mu'))$ with either $k\in \underline{\sf cont}(P')$ or $k\in \underline{\sf cont}(Q')$ but not both. In the first case, using induction and the dilation homomorphism we have
$$(\pm i) D^{\mu -P-Q}_{\mu-P}D^{\mu - P}_\mu = (-1)^{b(\langle P'\cup Q' \rangle_\mu')-b(Q')} (\mp i) D^{\mu-P-Q}_{\mu - \langle r,c\rangle_\mu}D^{\mu - \langle r,c\rangle_\mu}_\mu.$$ Noting that $b(\langle P\cup Q \rangle_\mu) = b(\langle P'\cup Q'\rangle_{\mu'})+1$ and $b(Q)=b(Q')$ gives the result.
In the second case, using induction and the dilation homomorphism we have 
$$(\pm i) D^{\mu -P-Q}_{\mu-P}D^{\mu - P}_\mu = (-1)^{b(\langle P'\cup Q' \rangle_\mu')-b(Q')} (\pm i) D^{\mu-P-Q}_{\mu - \langle r,c\rangle_\mu}D^{\mu - \langle r,c\rangle_\mu}_\mu.$$ Noting that $b(\langle P\cup Q \rangle_\mu) = b(\langle P'\cup Q'\rangle_{\mu'})+1$ and $b(Q)=b(Q')+1$ gives the result.

 \begin{figure}[ht!]
$$ \begin{minipage}{5.5cm}
  \end{minipage}
$$
\caption{Rewriting the top left diagram, $D^{\mu-{\color{magenta}P}-{\color{cyan}Q}}_{\mu-{\color{magenta}P}}D^{\mu-{\color{magenta}P}}_\mu$, as the bottom diagram, 
$(-1)^{{\color{darkgreen}\langle{P}\cup {Q}\rangle_\mu}-1}
D^{\mu-{\color{magenta}P}-{\color{cyan}Q}}_
{\mu- {\color{darkgreen}\langle{P}\cup {Q}\rangle_\mu}}
D^{\mu- {\color{darkgreen}\langle{P}\cup {Q}\rangle_\mu}}
_\mu
$.
}
\label{a big one}
\end{figure}

\bigskip
It remains to show that these form a complete list of relations. 
\smallskip

\noindent
{\bf Completeness of relations. }
It is enough to show that, using (\ref{rel1})--(\ref{adjacent}) and their duals, we can rewrite any product of $k$ degree $1$ generators as a linear combination of light leaves basis elements. We proceed by induction on $k$. For $k=1$ there is nothing to prove. For $k=2$, note that (\ref{rel1})--(\ref{adjacent}) and their duals cover precisely all possible (non-zero) product of two degree $1$ generators and rewrite these as linear combinations of basis elements. Now, assume that the result holds for $k$ and consider a product of $k+1$ generators. By induction, it is enough to consider a product of the form
$$D^\mu_\la D^\la_\nu D^\nu_{\nu \pm P}$$
where $D^\mu_\la D^\la_\nu$ is a basis element of degree $k$ and $P$ is a Dyck path. To show that this product can be rewritten as a linear combination of basis elements, we will additionally use induction on $\ell(\la)+\ell(\nu)$. If $ \ell(\la)+\ell(\nu) = 0$ then $\la = \nu = \emptyset$, $\nu +P = (1)$ and we have, using (\ref{rel1}),  
$$D^\mu_\emptyset D^\emptyset_\emptyset D^\emptyset_{(1)} = D^\mu_\emptyset  D^\emptyset_{(1)}$$
which is a basis element. Now assume that $\ell(\la) + \ell(\nu) \geq 1$. If $\la \neq \nu$ then we can write 
$$D^\la_\nu = D^\la_{\nu -Q} D^{\nu - Q}_\nu$$
for some $Q\in {\rm DRem}(\nu)$. Now using (\ref{rel1})--(\ref{adjacent}) we can write
$$D^{\nu - Q}_\nu D^\nu_{\nu \pm P} = \sum_{\nu'}c_{\nu'}D^{\nu - Q}_{\nu'}D^{\nu'}_{\nu \pm P}$$
for some $c_{\nu '}\in \Bbbk$ where $\ell(\nu') \leq \ell(\nu - Q) < \ell(\nu)$. Now we have
\begin{eqnarray*}
D^\mu_\la D^\la_\nu D^\nu_{\nu \pm P} &=& D^\mu_\la D^\la_{\nu -Q}D^{\nu - Q}_\nu D^\nu_{\nu \pm P} \\
&=& \sum_{\nu '} c_{\nu'} (D^\mu_\la D^\la_{\nu - Q}D^{\nu - Q}_{\nu'})D^{\nu'}_{\nu \pm P}\\
&=& \sum_{\nu' , \la'} d_{\nu ' , \la'} D^\mu_{\la'}D^{\la'}_{\nu'}D^{\nu'}_{\nu\pm P}.
\end{eqnarray*}
by induction, and $\ell(\la')+\ell(\nu')<\ell(\la)+\ell(\nu)$ so we're done. It remains to consider the case where $\la = \nu$. Here we must have $\mu \neq \la$. First observe that 
$$D^\mu_\la D^\la_\la D^\la_{\la +P} = D^\mu_\la D^\la_{\la +P}$$
by (\ref{rel1}) and this is a basis element. The last case to consider is
$$D^\mu_\la D^\la_\la D^\la_{\la -P} = D^\mu_\la D^\la_{\la -P}.$$
As $\mu\neq \la$ we have 
$D^\mu_\la = D^\mu_{ \la +Q}D^{\la + Q}_\la$, for some $Q\in {\rm DAdd}(\la)$ and so
$$D^\mu_\la D^\la_\la D^\la_{\la +P} = D^\mu_{\la +Q}D^{\la +Q}_\la D^\la_{\la -P}.$$
Now, using (\ref{rel1})--(\ref{adjacent}) we have 
$$D^{\la +Q}_\la D^\la_{\la -P} = \sum_{\nu'}c_{\nu'} D^{\la + Q}_{\nu'} D^{\nu'}_{\la -P}$$
with $\ell(\nu') \leq \ell (\la - P)<\ell(\la)$ and so 
\begin{eqnarray*}
D^\mu_{\la +Q}D^{\la +Q}_\la D^\la_{\la -P} &=& \sum_{\nu'}c_{\nu'} (D^\mu_{\la + Q} D^{\la + Q}_{\nu'}) D^{\nu'}_{\la -P}\\
&=& \sum_{\la' , \nu'}d_{\la' , \nu'} D^\mu_{\la '} D^{\la'}_{\nu'}D^{\nu'}_{\la - P}
\end{eqnarray*}
using induction as ${\rm deg} D^\mu_{\la + Q} D^{\la + Q}_{\nu'} = k$, with $\ell(\nu') \leq \ell(\la -P)< \ell(\la)$. Now as $\ell(\la')+\ell(\nu')<\ell(\la)+\ell(\nu)$, we're done by induction.
 \end{proof}

\subsection{Recasting the Dyck presentation as a quiver and relations}
Gabriel proved  that every basic algebra is isomorphic to the path algebra of its ${\rm Ext} $-quiver modulo relations. 
We now go through the formal procedure of recasting the Dyck presentation in this language.  
%
%
%
%

\begin{defn}\label{quiverdefn}
We define the Dyck quiver $\mathscr{D}_{m,n}$ with    vertex set 
$\{E_\la \mid \la \in \mptn\}$  and 
  arrows   $d^\la_\mu: \la   \to  \mu$ 
and 
   $d^\mu_\la:\mu \to \la$ for every $\la=\mu - P$ with $P\in {\rm DRem}(\mu)$. 
\end{defn}

An example is depicted in \cref{quiver}.

\begin{figure}[ht!]
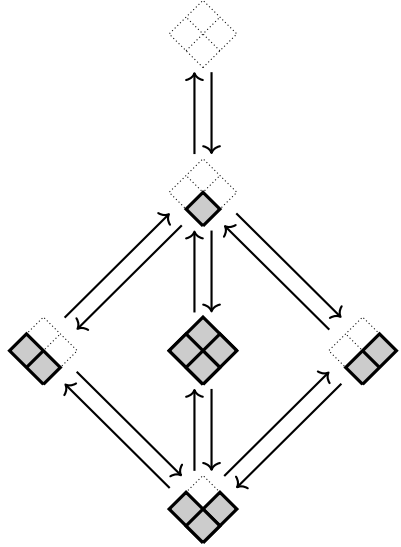


$$

 $$

 \caption{The  quiver $\mathscr{D}_{2,2}$. }
 \label{quiver}
 
 \end{figure}

\begin{prop}
The map 
  $$
   E_\mu  \mapsto {\sf 1} _\mu
\qquad d^\la_\mu \mapsto  D^\la_\mu $$
defines an algebra homomorphism from the path algebra of the Dyck quiver $\mathscr{D}_{m,n}$ to $\mathcal{H}_{m,n} $. 
Thus, the algebra  $\mathcal{H}_{m,n} $ is isomorphic to the quotient of the  path algebra of the Dyck quiver 
$\mathscr{D}_{m,n}$ by the quadratic relations given in (\ref{rel1})--(\ref{adjacent}) (where we replace all $D_\mu^\la$'s with $d_\mu^\la$'s).
\end{prop}
\begin{proof}
This follows directly from \cref{presentation}.
\end{proof}

\begin{eg}\label{building on}
Continuing with \cref{quiver} we have that  the algebra  $\mathcal{H}_{2,2}$ 
 is the path algebra of the quiver $\mathscr{D}_{2,2}$ 
modulo the following relations and their duals 
\begin{align}\label{1thing}
d^\varnothing _{(1)}d^{(1)}_{(2)}=0=d^\varnothing _{(1)}d^{(1)}_{(1^2)}
\qquad
d^{(1)}_{(2)}d^{(2)}_{(2,1)}
=
d^{(1)}_{(2^2)}d^{(2^2)}_{(2,1)}
=
d^{(1)}_{(1^2)}d^{(1^2)}_{(2,1)}
\qquad d^{(1)}_{\la}d_{(1)}^{\la}  
   =-d_\varnothing^{(1)}
d^\varnothing_{(1)} 
\end{align}
for $\la = (2), (1^2)$ or $(2^2)$,
\begin{align}
\label{2thing}
d^{(2,1)}_{(2^2)}d_{(2,1)}^{(2^2)}
   =-d^{(2,1)}_{(2 )}d_{(2,1)}^{(2 )}
   -d^{(2,1)}_{(1^2 )}d_{(2,1)}^{(1^2 )}
  \end{align}
   and 
  for any pair  $\nu<\mu$ not of the above form, we have that   
\begin{align}\label{3thing}
d^\nu _{\mu}
d^\mu_\nu
=0.
\end{align}
\end{eg}

\begin{eg}    \label{example-biserial}\label{coef-quiver}
 Apart from the categories of this paper, the only  Hecke categories   whose quiver and relations were understood were those corresponding to Weyl groups of ranks 2 and 3 \cite{MR2017061} 
 and those bi-serial algebras corresponding to $(W,P)=(S_n,S_{n-1})$, which we now describe.  
In this case, the quiver is depicted in \cref{quive2r}.

\begin{figure}[ht!]

$$

 $$

 \caption{The  quiver $\mathscr{D}_{n,1}$. }
 \label{quive2r}
 
 \end{figure}  
  
We have that  the algebra  $\mathcal{H}_{n,1}$ 
 is the path algebra of the quiver $\mathscr{D}_{n,1}$ 
modulo the following relations and their duals 
  $$
  d^{(k)}_{(k+1)}
    d_{(k)}^{(k+1)}
    =
  d^{(k)}_{(k-1)}        d_{(k)}^{(k-1)}
\qquad 
  d^{(k)}_{(k\pm1)}
    d_{(k\pm2)}^{(k\pm1)}=0
  $$
for $1\leq k <n$.  
The projective modules are all uni-serial or biserial and their structure is encoded in the  Alperin diagrams in \cref{quiver3}.

\begin{figure}[ht!]
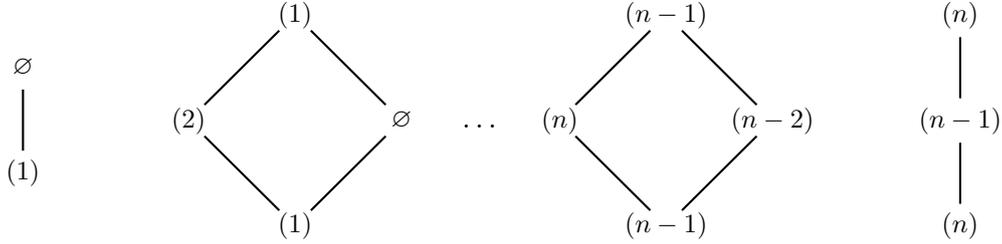


$$
 \begin{minipage}{1cm} %
\end{minipage}
$$

 \caption{The  Alperin diagrams of projective modules for 
 $\mathcal{H}_{n,1}$. }
 \label{quiver3}
 
 \end{figure}  
  
 \end{eg}

 \section{Submodule structure of standard modules}   
    
 For this section, we assume that $\Bbbk$ is a field. As noted in \cref{heere ris the basus}, the algebra $\mathcal{H}_{m,n}$ is a basic (positively) graded quasi-hereditary algebra with graded cellular basis given by $$\{D^\mu_\la D^\la_\nu \, : \, \text{for all Dyck pairs} \, (\la, \mu), (\la, \nu)\, \text{with}\, \, \la, \mu, \nu\in \mptn \}.$$ 
For $\la \in \mptn$, write $\mathcal{H}_{m,n}^{\leq \la} = {\rm span}\{D^\mu_\alpha D^\alpha_\la \, : \, \alpha, \mu  \in \mptn , \, \alpha \leq \la\}$ and $\mathcal{H}_{m,n}^{< \la} = {\rm span}\{D^\mu_\alpha D^\alpha_\la \, : \, \alpha, \mu  \in \mptn , \, \alpha < \la\}$. 
Setting $${\rm DP}(\la):=\{ \mu\in \mptn\, : \, (\la, \mu)\, \text{is a Dyck pair}\},$$ the (left) {\sf standard module} $\Delta_{m,n}(\la) = \mathcal{H}_{m,n}^{\leq \la} / \mathcal{H}_{m,n}^{< \la} $ has a basis given by 
 $$\{ u_\mu := D^\mu_\la + \mathcal{H}_{m,n}^{<\la} \,  : \, \mu \in {\rm DP}(\la)\}.$$  Each $u_\mu$ generates a submodule of $\Delta_{m,n}(\la)$ with a $1$-dimensional simple head, which we denote by $L_{m,n}(\mu)$. 
 
 In this section, we describe the full submodule structure of the standard modules. As $\mathcal{H}_{m,n}$ is positively graded, the grading provides a submodule filtration of $\Delta_{m,n}(\la)$. Decompose ${\rm DP}(\la)$ as
 $${\rm DP}(\la) = \bigsqcup_{k \geq 0} {\rm DP}_k(\la) \quad \text{where} \quad  {\rm DP}_k(\la) = \{\mu\in {\rm DP}(\la) \, : \, \deg (\la, \mu)=k\}.$$
 Note further, that the algebra $\mathcal{H}_{m,n}$ is generated in degree $1$. This implies that, in order to describe the full submodule structure, it is enough to find, for each $\mu \in {\rm DP}_k(\la)$, the set of all $\nu \in {\rm DP}_{k+1}(\la)$ such that 
 $$u_\nu = c  D^\nu_\mu u_\mu$$
 for some $c\in \Bbbk$ and $\nu = \mu \pm P$ for some $P\in {\rm DAdd}(\mu)$ or $P\in {\rm DRem}(\mu)$ respectively. Thus, the condition that $\nu = \mu \pm P$ for some $P\in {\rm DRem}(\mu)$ or $P\in {\rm DAdd}(\mu)$ respectively is certainly a necessary condition for the existence of an extension between $L_{m,n}(\mu)$ and $L_{m,n}(\nu)$ in $\Delta_{m,n}(\la)$. We claim that it is also sufficient. Assume $\mu \setminus \la = \sqcup_i Q^i$. For $P\in {\rm DAdd}(\mu)$, note that $(\la, \mu +P)$ is a Dyck pair if and only if $P$ is not adjacent to any $Q^i$ and so in this case $(\mu +P)\setminus \la = \sqcup_i Q^i \sqcup P$ is the Dyck tiling and we have 
 $$D^{\mu +P}_\mu D^\mu_\la = D^{\mu +P}_\la$$
 by the definition of the light leaves basis.
For $P\in {\rm DRem}(\mu)$, the only way to have ${\rm deg}(\la, \mu -P) = {\rm deg}(\la , \mu) +1$ is if $P\notin \{Q^i\}$ and there exists some $Q\in \{Q^i\}$ such that $P\prec Q$ do not commute. In this case we have $Q\setminus P = R \sqcup S$ for some $R, S\in {\rm DRem}(\mu - P)$. We prove by induction on ${\rm deg}(\la, \mu)$ that $D^{\mu- P}_\mu D^\mu_\la = D^{\mu - P}_\la$.  If ${\rm deg}(\la, \mu) = 1$ then $\mu - Q=\la$ and the non-commuting relation gives  
 $$D^{\mu -P}_\mu D^\mu_\la = D^{\mu -P}_{\mu - P - R} D^{\mu - P - R}_{\mu - Q} = D^{\mu - P}_\la$$
 as required. 
 
 Now assume that ${\rm deg}(\la, \mu)\geq 2$. 
Suppose $Q\not\prec Q'$ for  $Q'\in \{Q^i\}$. Then we can write $D^\mu_\la = D^\mu_{\mu - Q} D^{\mu - Q}_\la$ and we have
 $$D^{\mu - P}_\mu D^\mu_\la = D^{\mu - P}_\mu D^{\mu}_{\mu - Q}D^{\mu - Q}_\la = D^{\mu - P}_{\mu - P - R} D^{\mu - P - R}_{\mu - Q} D^{\mu - Q}_\la = D^{\mu - P}_\la$$
 by the non-commuting relations and the definition of the light leaves basis. 
Otherwise,
we have that   $D^{\mu}_\la = D^\mu_{\mu - Q'}D^{\mu - Q'}_\la$ for some $Q'$ commuting with $P$. Then we have
 $$D^{\mu - P}_\mu D^\mu_\la = D^{\mu - P}_\mu D^{\mu}_{\mu - Q'}D^{\mu -Q'}_\la = D^{\mu - P}_{\mu - P - Q'} D^{\mu - P - Q'}_{\mu - Q'}D^{\mu - Q'}_\la = D^{\mu - P}_{\mu - P - Q'}D^{\mu - P - Q'}_\la = D^{\mu - P}_\la$$
where the second equality follows from the commuting relation, the third one follows by induction (as ${\rm deg}(\la, \mu - Q') = {\rm deg}(\la, \mu) - 1$), and the final equality follows by  the definition of the light leaves basis. 
 
  \begin{rmk}\label{socle} We set $k_\la = \max \{ k\geq 0 \, | \, {\rm DP}_k(\la)\neq \emptyset\}$. Then it is easy to check that ${\rm DP}_{k_\la}(\la)$ consists of a single element $\mu_\la$.  To construct the cup diagram of $\mu_\la$, start with the weight $\la$ and apply the following two steps:
 
\begin{enumerate}[leftmargin=*]
 \item  
repeatedly find a pair of vertices labeled $\up$ $\down$ in order from left to right that are neighbours in the sense that there are only vertices already joined by cups in between. Join these new vertices together with a cup. Then repeat the process until there are no more such $\up$ $\down$ pairs. 

We are left with a sequences of $\down$'s followed by a sequences of $\up$'s.
\item Join these using concentric anti-clockwise cups.  We are left with either a sequence of $\up$'s or a sequence of $\down$'s. Draw vertical rays on these. 
\end{enumerate}

 \end{rmk}

Suppose $\mu_\la \setminus \la = \sqcup_i Q^i$. 
Note that $\mu_\la$ is characterised by the following two properties:
 \begin{enumerate}
 \item There is no $P\in {\rm DAdd}(\mu_\la)$ such that $P \sqcup \left(\bigsqcup_i Q^i\right)$ is a Dyck tiling of $(\mu_\la+P) \setminus \la$.
 \item If $P\in \{Q^i\}$ and $Q\prec P$ then $Q\in \{Q^i\}$.
 \end{enumerate}
 
 This implies, in particular, that if $\mu \in {\rm DP}_k(\la)$ for $k<k_\la$, then either (1) or (2) above fails and we have seen that in each case we can find some $h\in \mathcal{H}_{m,n}$ and $\nu\in {\rm DP}_{k+1}(\la)$ such that $u_\nu = hu_\mu$.  Thus the radical and socle filtration of $\Delta_{m,n}(\la)$ coincide with its grading filtration and the socle of $\Delta_{m,n}(\la)$ is given by $L_{m,n}(\mu_\la)$.
 We have proved the following:

\begin{thm}\label{submodule}
Let $\la \in \mptn$. The Alperin diagram of the standard module $\Delta_{m,n}(\la)$ has vertex set labelled by the set 
 $\{ L_{m,n}(\mu)\, : \,  \mu\in {\rm DP}(\la)\}$  and edges
$$L_{m,n}(\mu) \longrightarrow L_{m,n}(\nu)$$
whenever $\mu \in {\rm DP}_k(\la)$, $\nu \in {\rm DP}_{k+1}(\la)$ for some $k\geq 0$ and $\nu = \mu \pm P$ for some $P\in {\rm DAdd}(\mu)$ or $P\in {\rm DRem}(\mu)$ respectively. 
Moreover, the radical and socle filtration both coincide with the grading filtration and $\Delta_{m,n}(\la)$ has simple socle isomorphic to $L_{m,n}(\mu_\la)$ (where $\mu_\la$ is described in \cref{socle}).
\end{thm}

An example is provided in \cref{standard21}.

\begin{figure}[ht!]
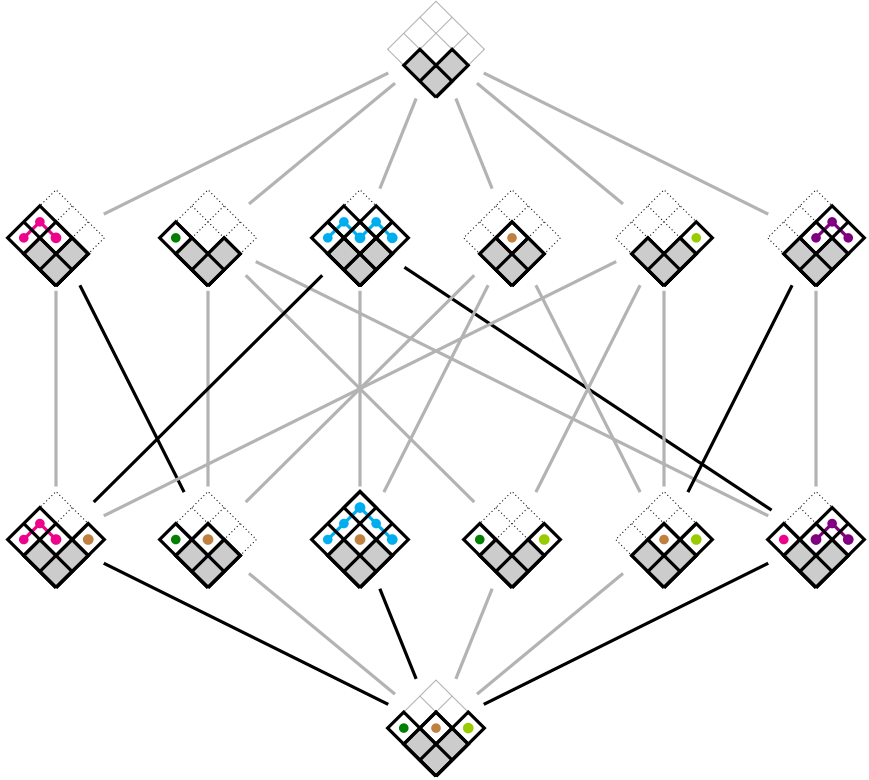

$$

 $$     
\caption{The Alperin diagram for the standard module $\Delta_{3,3}(2,1)$.  We    use    grey lines to indicate pairs obtained by adding a Dyck path and black lines to indicate the pairs obtained by removing a Dyck path.
}\label{standard21}
\end{figure}

     \section{The isomorphism between Hecke categories and  Khovanov arc algebras}

We now utilise  our newfound  presentations  in order to prove that the Khovanov arc algebras and   Hecke categories are isomorphic as $\ZZ$-graded $\Bbbk$-algebras for $\Bbbk$ any commutative integral domain containing a square root of $-1$.

   \subsection{Signs and the statement of the isomorphism}
For the purpose of defining our isomorphism,   we will wish to consider all degree 1 diagrams in the Khovanov arc algebra. 
 Using the dilation  homomorphism of section 6, we are often able to restrict our 
attention to diagrams which are incontractible. which by \cref{incontr-prtns} are of the form $ \underline{ \mu} { \la}  
  \overline{ \la}$ (or its dual) for $\mu$ a rectangle and $\lambda=\mu -P$ with $b(P)=1$.
 The following lemma is immediate by construction of the arc diagrams.
 
 \begin{lem} 
The diagrams $\underline{\mu}\la \overline{\la}$ for $(\mu,\la) = ((1), \emptyset)$, $(\mu,\la) =( (c),  (c-1))$ and $(\mu,\la) =
( (1^r),  (1^{r-1}))$ are given respectively by
  the   {\sf arc, left-zigzag} and  {\sf right-zigzag} 
 diagrams 
\begin{equation}\label{gens-arc0} 
 
\end{equation}
with $c-2$ (respectively $r-2$) vertical strands to the right (respectively, the left).
For $(\mu,\la) =( (c^r),  (c^{r-1}, c-1))$ with $r\geq c>1$ we have that $ \underline{\mu} { \la}  
  \overline{\la}$  is the  
  {\sf brace  generator} diagram  
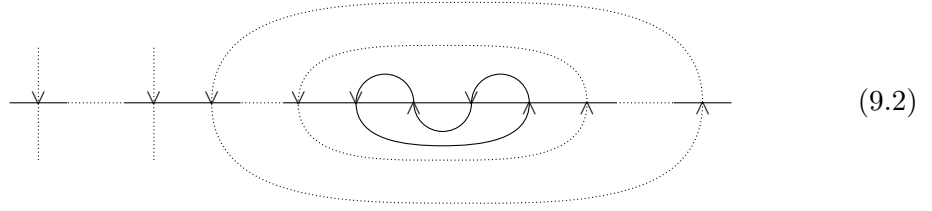
\begin{equation}\label{gens-arc1}  \begin{minipage}{8cm}
 %
 \end{minipage}
\end{equation}
with $c-2$ concentric circles and  a total of $r-c$ vertical strands to the left of the diagram.  The case with 
$c\geq r \geq 1$ is similar but with $r-2$ 
concentric circles
and  a total of $c-r$ vertical strands to the right of the diagram.

 \end{lem}

\begin{rmk}
The trivial embeddings $\mptn \rightarrow \mathscr{P}_{m+1,n}$ and $\mptn \rightarrow \mathscr{P}_{m,n+1}$ sending a partition to itself extends to algebra embeddings $\mathcal{K}_{m,n} \rightarrow \mathcal{K}_{m+1, n}$ and $\mathcal{K}_{m,n} \rightarrow \mathcal{K}_{m, n+1}$ defined on the arc diagrams by adding an upwards strand to the left or a downwards strand to the right. We have chosen to represent each arc diagram $\underline{\mu}\la\overline{\nu}$ in the smallest $\mathcal{K}_{m,n}$ where it is defined to avoid drawing lots of vertical strands which play no role in the multiplication or the next definition.
\end{rmk}

\begin{defn}\label{sign-defn}
Let $(\la, \mu)$ be a Dyck pair of degree $1$. Then $\la = \mu - P$ for some $P\in {\rm DRem}(\mu)$. We set ${\sf sgn}(\la, \mu)$ to be the average of the elements in the set $\underline{\sf cont}(P)$. In other words if the unique clockwise cup in $\underline{\mu}\la\overline{\la}$ connects vertices in position $i-\tfrac{1}{2}$ and $j+\tfrac{1}{2}$ for $i\leq j$ then ${\sf sgn}(\la, \mu) = \tfrac{1}{2}(j+i)$
   \end{defn}
   
\begin{eg} The    generators   $ \underline{ \mu} { \la}  
  \overline{  \la}$
 of the form 
     $$

$$
for $\mu=(1^1)$, $(2)$, $(1^2)$, $(2^2)$, $(3)$ respectively    and $\la = \mu - P$ with $b(P) = 1$  have signs $0, -1, 1, 0$ and $-2$.

\end{eg}

\begin{thm}\label{isomorphismthem}

We have a graded $\Bbbk$-algebra isomorphism $\Psi: \mathcal{H}_{m,n }   \to \mathcal{K}_{m,n}$ defined on generators by setting, for all $\la,\mu\in \mptn$ such that  $(\la, \mu)$  is a Dyck pair of degree $1$,
$$
\Psi({\sf 1}_\la)  = 
\underline{ \la}  { \la} 
 \overline{ \la},
\qquad
\Psi (D^\la_\mu ) = 
i^{{\sf sgn}( \la, \mu)}
 \underline{ \la} { \la}  
  \overline{  \mu}\qquad  \Psi (D^\mu_\la ) = 
i^{{\sf sgn}( \la, \mu)}
 \underline{ \mu} { \la}  
  \overline{  \la}
$$
where  $i$ is a square root of $-1$ in $\Bbbk$. 
\end{thm}

\subsection{Proof of the isomorphism}The remainder of this section is dedicated to the proof of \cref{isomorphismthem}.

 \begin{lem}
 [Local idempotent relations]\label{local-surg}
  Any anticlockwise-oriented circle which intersects the weight at {\em precisely} two points is a local weight-idempotent in the following sense:
applying  the local surgery procedure at this point is equivalent to simply
 deleting this circle (see \cref{local-idem} for an example). 
  \end{lem}

 \begin{proof}
 This follows immediately using the surgery procedures $1\otimes 1 \mapsto 1$, $1\otimes x \mapsto x$ and $1\otimes y \mapsto y$.
 \end{proof}

\begin{figure}[ht!]
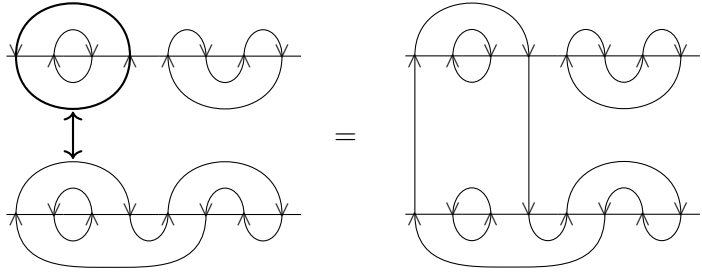

  \begin{equation*} 
 \quad   \begin{minipage}{4.5cm}
 \end{minipage}
 \end{equation*} 
\caption{An indicative example of the local idempotent multiplication. Notice that the righthand-side can be thought of as being obtained by deleting the bold circle (together with its weight arrows) and ``stretching" the bottom component so that it reaches the top.
}
\label{local-idem}
\end{figure}

 \begin{prop}[The idempotent relations]
 The idempotent relations are preserved by $\Psi$. 
 \end{prop}
 
\begin{proof}
Note that the element $\underline{\la}\la\overline{\la}$ contains only anticlockwise circles intersecting the weight at precisely two points. Thus the result follows from \cref{local-surg}.
\end{proof}
    
    \begin{prop}\label{firstrel}(The self-dual relation)
Let  $P\in {\rm DRem}(\mu)$ and $\la = \mu - P$. Then we   have 
\begin{eqnarray*}  (-1)^{\sgnlamu}
  \underline{\la} \la  \overline{\mu}
  \cdot  \underline{\mu} \la  \overline{\la}
= &&
2
\!\!  \sum_{   \begin{subarray}{c} 
\nu= \la- Q \\ P\prec Q \end{subarray}}
 \!
(-1) ^{b(Q)+b(P)-1+ {  \sgnnula 	} }  \underline{\la} \nu  \overline{\nu}
  \cdot  \underline{\nu} \nu \overline{\la}\\ 
  && + \!
     \sum_{   \begin{subarray}{c} 
\nu= \la- Q \\ P \text { adj. } Q \end{subarray}}
 \!\!\!\!\!
(-1) ^{b(Q)+b(P)-1+ {  \sgnnula 	} }  \underline{\la} \nu  \overline{\nu}
  \cdot  \underline{\nu} \nu \overline{\la} .  
\end{eqnarray*}
  \end{prop}

 \begin{proof} We first focus on the incontractible   case in which $\mu=(c^r)$ is a rectangle with $r,c\geq 1$ and $b(P)=1$ so that $\la=(c^{r-1},c-1)$.  
 
 \noindent {\bf The generic case. } We first consider the case  $r,c>1$ and set $m = \min \{ r,c\}$. In this case, the set of all Dyck paths $Q\in {\rm DRem}(\la)$ appearing on the right hand side of the equation can be described as follows. 
 The set 
 $$\{ Q\in {\rm DRem}(\la) \, : \, Q \, \text{ adj. }  \, P\} = \{Q^{-1}, Q^1\}$$
satisfies $b(Q^{\pm 1}) = 1$ and, setting $\nu_{\pm 1} = \la - Q^{\pm 1}$, we have ${\sf sgn}(\nu_{\pm 1}, \la) = {\sf sgn}(\la, \mu) \pm 1$.
The set 
$$\{Q\in {\rm DRem}(\la) \, : \, P\prec Q\} = \{Q^2, \ldots , Q^{m-1}\}$$
satisfies $b(Q^x) = x+1$ for $2\leq x\leq m-1$ and, setting $\nu_x = \la - Q^x$, we have ${\sf sgn}(\nu_x, \la) = {\sf sgn}(\la, \mu)$.  An example is given in Figure \ref{sum-relation}.

Thus, the self-dual relation can be rewritten in this case as
\begin{equation}\label{selfdual-rectangle}
\underline{\la}\la \overline{\mu} \cdot \underline{\mu}\la \overline{\la} = \underline{\la}\nu_{-1} \overline{\nu_{-1}} \cdot \underline{\nu_{-1}}\nu_{-1} \overline{\la} +\underline{\la}\nu_{1} \overline{\nu_{1}} \cdot \underline{\nu_{1}}\nu_{1} \overline{\la} + 2 \sum_{x=2}^{m-1} (-1)^{x+1} \underline{\la}\nu_{x} \overline{\nu_{x}} \cdot \underline{\nu_{x}}\nu_{x} \overline{\la}.
\end{equation}

The weight diagram for $\mu$ consists of $r$ down arrows followed by $c$ up arrows; the weight diagram for $\la=(c^{r-1},c-1)$ is obtained from that of $\mu$ by swapping the neighbouring $\down \up$.  Pictorially
$$
 \mu=
\left(\begin{minipage}{5 cm}
\end{minipage}\right)$$
and therefore $  \underline{\mu} \la  \overline{\la}
$ can be pictured as follows,
$$
  \underline{\mu} \la  \overline{\la}
  =\;
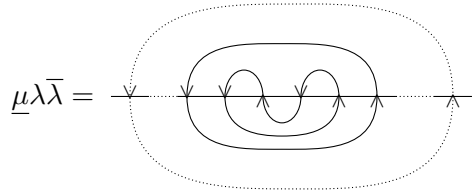
   \begin{minipage}{5cm}
 %
 \end{minipage}
  $$ plus $|r-c|$ vertical strands as in \cref{gens-arc1} on either side. 
In other words, $  \underline{\mu} \la  \overline{\la}
$ consists of a brace  surrounded by $m-2$ concentric anti-clockwise circles.  
We will ignore the vertical strands  as they have no effect on the multiplication.

Performing surgery on 
 $ \underline{\la} \la  \overline{\mu}
\cdot 
 \underline{\mu} \la  \overline{\la}$ 
we first apply \cref{local-surg} to the $m-2$ concentric circles one at a time (starting from the outermost pair of circles) until we obtain the diagram
$$
   \underline{\la} \la  \overline{\mu}\cdot 
 \underline{\mu} \la  \overline{\la}
  = \;\;
   \begin{minipage}{6cm}
 \end{minipage}
  $$ 
We then apply the two surgery steps detailed in  \cref{brace-surgery} 
to the innermost pair of braces to obtain a sum of two diagrams
\begin{equation}\label{LMS}
 \begin{minipage}{5cm}
 %
 \end{minipage}
\end{equation}
We need to compare this to the righthand-side of the equation in  \cref{selfdual-rectangle}.  
The diagrams 
$ \underline{\nu_1} \nu_1  \overline{\la}$  and $ \underline{\nu_{-1}} \nu_{-1}  \overline{\la}$  are equal to 
\begin{equation}\label{eqna1} 
 \end{minipage}
 \end{equation}
respectively.
Using \cref{local-surg} and  \cref{brace-surgery2}, we have that 
the product  $  \underline{\la} \nu_1  \overline{\nu_1}
  \cdot  \underline{\nu_1} \nu_1 \overline{\la}$  is given by 
  \begin{equation}\label{LMS2}
    \begin{minipage}{4.5cm}
 %
 \end{minipage} 
\end{equation}
Similarly, the product $  \underline{\la} \nu_{-1}  \overline{\nu_{-1}}
  \cdot  \underline{\nu_{-1}} \nu_{-1} \overline{\la}$ is given by

  \begin{equation}\label{LMS3}
 \begin{minipage}{4.5cm}
 %
 \end{minipage} 
\end{equation}
where we have highlighted the pair of  re-oriented circles in each case (with the pink circle of degree 2 and the blue of degree 0).  Notice that the sum of the left-hand terms in each of 
(\ref{LMS2}) and (\ref{LMS3}) is the required sum (\ref{LMS}). The right hand term in (\ref{LMS2}) and (\ref{LMS3}) are identical, and we denote this diagram by $D_1$. We will show that the sum of these 2 terms (which is equal to $2D_1$) will cancel with the remaining terms in the larger sum.  
For $2\leq x\leq m-2$ we have 
\begin{equation}\label{eqna3}   \underline{\nu_x} \nu_x  \overline{\la}
=
    \begin{minipage}{9cm}
 \end{minipage}
 \end{equation}
where there are $x-2$ concentric dotted  hollow circles in the middle of the diagram.

Therefore  we have that  $\underline{\la} \nu_x  \overline{\nu_x}
  \cdot  \underline{\nu_x} \nu_x \overline{\la}$  is equal to 
\begin{align*}
\label{eqna4}   
   \begin{minipage}{6.8cm}  
 %
 \end{minipage}\;\;\;.
 \end{align*}
   We denote the first diagram in the sum by $D_{x-1}$ and the second by $D_x$, so $D_x$ is the diagram where the pink anticlockwise circle is at distance $x$ from the small innermost circles. (Note that this is consistent with our notation for $D_1$ above). 
 Finally, for $x=m-1$ we have 
\begin{equation*}
  \underline{\nu_{m-1}} \nu_{m-1}  \overline{\la}
=
    \begin{minipage}{7cm}
 \end{minipage}
 \end{equation*}
in which only the  outermost strand is clockwise oriented.  
This gives $\underline{\la} \nu_{m-1}  \overline{\nu_{m-1}}
  \cdot  \underline{\nu_{m-1}} \nu_{m-1} \overline{\la}$  is equal to
  $$
    \begin{minipage}{7cm}
 \end{minipage}$$
which is equal to $D_{m-2}$.       Replacing all terms into  the right-hand side of (\ref{selfdual-rectangle})  we obtain 
$$(\underline{\la} \la\overline{\mu}\cdot \underline{\mu}\la \overline{\la} +2 D_1 ) + 2 \sum_{x=2}^{m-2} (-1)^{x+1} (D_{x-1}+D_x) + 2 (-1)^{m} D_{m-2} = \underline{\la} \la\overline{\mu}\cdot \underline{\mu}\la \overline{\la}$$
as required.

 \bigskip
\noindent {\bf Degenerate cases. } We now consider the degenerate cases 
in which   $r$ or $c$ is equal to 1.   
The $r=1=c$ case simply follows as 
$$
  \underline{\la} \la  \overline{\mu}
  \cdot  \underline{\mu} \la  \overline{\la}=
\begin{minipage}{1.2 cm}
 \end{minipage}=0,$$
 as required. 
 It remains to consider the cases $\mu  = (c), \la = (c-1)$ and $\mu = (1^r), \la = (1^{r-1})$. We deal with the first one; the second one is similar. Here the weight $\mu $ and $\la$  are  of the form  
 $$
 \mu=
\left(\begin{minipage}{3.5cm}
 %
\end{minipage}\right)$$
So we have that  $$
   \underline{\mu} \la  \overline{\la} =
   \begin{minipage}{3cm}
 %
 \end{minipage}\;\;\;\;\;.
$$
 Note that here $\la$ only has one removable Dyck path, $Q^1$, which is adjacent to $P$, and using the same notation as in the generic case, we have
$$
  \underline{\nu_1} \nu_1  \overline{\la} =
   \begin{minipage}{3.5cm}
 %
 \end{minipage}
$$
for $c\geq 3$ or $c=2$ respectively. Thus, for $c\geq 3$ we have 
 $ \underline{\la} \la  \overline{\mu}
  \cdot  \underline{\mu} \la  \overline{\la}$ is equal to 
$$
 \begin{minipage}{3.2cm}
 %
 \end{minipage}$$
 which is equal to 
$ \underline{\la} \nu_1  \overline{\nu_1}
  \cdot  \underline{\nu_1} \nu_1 \overline{\la} $ as required.  The $c=2$ case can be handled in an identical fashion (except that we stop after the first equality and note that there are no up-strands on the right of the diagrams).
  
  \medskip
  
  Finally, take $(\la, \mu)$ with $\la = \mu -P$ to be any Dyck pair of degree 1 which is contractible at $k$, for some $k$. Then there exists $(\la', \mu')$ with $\la' = \mu' - P'$ such that $(\la, \mu) = (\varphi_k(\la'), \varphi_k(\mu'))$. For each $Q'\in {\rm DRem}(\la')$, $\nu' = \la ' - Q'$ we write $\nu = \varphi_k(\nu') = \la - Q$. We can assume by induction that the self-dual relation holds for $(\la', \mu')$ and applying the dilation homomorphism we obtain
\begin{equation}  \begin{array}{lll}  (-1)^{{\sf sgn}(\la',\mu')}
  \underline{\la} \la  \overline{\mu}
  \cdot  \underline{\mu} \la  \overline{\la}
& = &
2
 \displaystyle\sum_{   \begin{subarray}{c} 
\nu= \la- Q \\ P\prec Q \end{subarray}}

(-1) ^{b(Q')+b(P')-1+ {  {\sf sgn}(\nu' , \la')}	}   \underline{\la} \nu  \overline{\nu}
  \cdot  \underline{\nu} \nu \overline{\la}\\ 
  && \ + 
 \displaystyle     \sum_{   \begin{subarray}{c} 
\nu= \la- Q \\ P \text { adj. } Q \end{subarray}}

(-1) ^{b(Q')+b(P')-1+ {  {\sf sgn}(\nu' , \la')	} }  \underline{\la} \nu  \overline{\nu}
  \cdot  \underline{\nu} \nu \overline{\la}
  \end{array}\end{equation}\label{primed}
 First consider the case where $k\in \underline{\sf cont}(P')$. Then we have ${\sf sgn}(\la', \mu' ) = {\sf sgn}(\la , \mu)$ and $b(P') = b(P)-1$. For each $Q$ such that $P\prec Q$, we have $b(Q') = b(Q)-1$ and ${\sf sgn}(\nu', \la') = {\sf sgn}(\nu , \la)$ so we have
 \begin{eqnarray*}
 (-1)^{b(Q')+b(P')-1+{\sf sgn}(\nu', \la')}  &=&  (-1)^{b(Q)-1+b(P)-1-1+{\sf sgn}(\nu, \la)}\\
 &=&  (-1)^{b(Q)+b(P)-1+{\sf sgn}(\nu, \la)}.
 \end{eqnarray*}
 On the other hand, for each $Q$ adjacent to $P$, we have $b(Q') = b(Q)$ and ${\sf sgn}(\nu', \la') = {\sf sgn}(\nu , \la)\pm 1$ and so we have
 \begin{eqnarray*}
 (-1)^{b(Q')+b(P')-1+{\sf sgn}(\nu', \la')}  &=&  (-1)^{b(Q)+b(P)-1-1+{\sf sgn}(\nu, \la) \pm 1}\\
 &=&  (-1)^{b(Q)+b(P)-1+{\sf sgn}(\nu, \la)}.
 \end{eqnarray*}
So we obtain the required relation for $(\la, \mu)$.  

It remains to consider the case where $k\notin \underline{\sf cont}(P')$. In this case we have ${\sf sgn}(\la', \mu' ) = {\sf sgn}(\la , \mu)\pm 1$ and $b(P') = b(P)$. Let $Q'\in {\rm DRem}(\la')$. If $k\notin \underline{\sf cont}(Q')$ then we have $b(Q') = b(Q)$ and ${\sf sgn}(\nu', \la' ) = {\sf sgn}(\nu , \la)\pm 1$ and so we have 
$$ (-1)^{b(Q')+b(P')-1+{\sf sgn}(\nu', \la')}  =  (-1)^{b(Q)+b(P)-1+{\sf sgn}(\nu, \la)\pm 1}.$$
If $k\in \underline{\sf cont}(Q')$ then we have $b(Q') = b(Q)-1$ and ${\sf sgn}(\nu', \la' ) = {\sf sgn}(\nu , \la)$ and so we have 
$$ (-1)^{b(Q')+b(P')-1+{\sf sgn}(\nu', \la')}  =  (-1)^{b(Q)-1+b(P)-1+{\sf sgn}(\nu, \la)}.$$
Thus dividing the equation (\ref{primed}) by $-1$ gives the required equation for $(\la, \mu)$.
 \end{proof}

\begin{prop}[The commuting relations]
Let $P,Q\in {\rm DRem}(\mu)$ with $P\prec Q$ which commute, and let $\la=\mu-P$, $\nu=\mu-Q$, and $\alpha=\mu-P-Q$.
We have that
\begin{align}
i^{\sgnalla+\sgnlamu} \underline{\alpha} \alpha \overline{\la} \cdot \underline{\la} \la \overline{\mu} & = i^{\sgnalnu+\sgnnumu} \underline{\alpha} \alpha \overline{\nu} \cdot \underline{\nu} \nu \overline{\mu} \label{distcomm1} \\
i^{\sgnlamu+\sgnnumu} \underline{\nu} \nu \overline{\mu} \cdot \underline{\mu} \la \overline{\la} & = i^{\sgnalla+\sgnalnu} \underline{\nu} \alpha \overline{\alpha} \cdot \underline{\alpha} \alpha \overline{\la}. \label{distcomm2}
\end{align}
\end{prop}    

\begin{proof}
First note that ${\sf sgn}(\la, \mu) = {\sf sgn}(\alpha , \nu)$ as $\la = \mu - P$ and $\alpha = \nu - P$. Similarly, ${\sf sgn}(\alpha, \la) = {\sf sgn}(\nu , \mu)$. Thus we can cancel the signs on both sides of the equation.  
Now, if $P$ and $Q$ are distant then the result follows directly using \cref{local-surg}. It remains to consider the case where $P\prec Q$ and they commute.  
 We first focus on the  incontractible case in which     $\mu=(c^r)$ is a rectangle for $r,c>2$ and  $b(P)=1$ and $3 \leq b(Q) \leq \min\{r,c\}$ (note that we must have $b(Q)>2$ in order to commute with $P$).  
Set $m = \min \{r,c\}$ and assume that  $b(Q)=m$. 
We have that $\la=(c^{r-1},c-1)$, $\nu=((c-1)^{r-1}, r-m)$, and $\alpha=((c-1)^{r-2},c-2, r-m)$.
 Thus we have $\underline{\nu}\nu \overline{\mu} \cdot \underline{\mu}\la \overline{\la}$ is equal to 
 \begin{equation}\label{distcomm}
   \begin{minipage}{5.75cm}
 \end{minipage} 
 \end{equation}
while $\underline{\nu}\alpha \overline{\alpha} \cdot \underline{\alpha}\alpha \overline{\la}$ is equal to
\begin{equation}
   \begin{minipage}{5.75cm}
 %
 \end{minipage} 
\end{equation}
In both equations there are $r-3$  dotted anti-clockwise circles in each diagram; and equality follows from  applying  the $1\otimes 1 \mapsto 1$ rule a total of $r$ times as required. 
A very similar calculation proves \cref{distcomm1}.
The other cases, where $3\leq b(Q) \leq m-1$ are completely analogous, except that the large arc in $\underline{\mu} \nu \overline{\nu}$ and $\underline{\la} \alpha \overline{\alpha}$ forms part of a zigzag or a brace.

Finally, the general case follows directly by applying the dilation homomorphism.
\end{proof}


    \begin{prop}[The non-commuting relation]\label{almostthere}
Let $P,Q\in {\rm DRem}(\mu)$ with $P\prec Q$ which do not commute, and let $\la=\mu-P$ and $\nu=\mu-Q$. Then  $Q\setminus P = Q^1 \sqcup Q^2$ where $Q^1, Q^2 \in {\rm DRem}(\mu - P)$ and we
set $\alpha=\la-Q^1$ and $\beta=\la-Q^2$. 
We  have that 
%
%
%
\begin{equation}\label{thisoney}
  i^{\sgnlamu+\sgnnumu}  \underline{\nu} \nu  \overline{\mu}
  \cdot  \underline{\mu} \la  \overline{\la} 
  =
  %
  %
  %
  i^{ 
\sgnalla+\sgnnual
 } 
  \underline{\nu} \nu  \overline{ \alpha}
  \cdot  \underline{ \alpha}  \alpha \overline{\la}
    = 
  i^{  
{\sf sgn}(\beta,\la)+{\sf sgn}(\nu,\beta)
 } 
  \underline{\nu} \nu  \overline{ \beta}
  \cdot  \underline{ \beta}  \beta \overline{\la}.
\end{equation}

\end{prop}

\begin{proof}
We start by proving that all the signs in the equation are equal. Assume that $P$ corresponds to a cup connecting $i_P -\tfrac{1}{2}$ and $j_P +\tfrac{1}{2}$ and $Q$ corresponds to a cup connecting $i_Q -\tfrac{1}{2}$ and $j_Q +\tfrac{1}{2}$. Then $Q^1$ corresponds to a cup connecting $i_Q - \tfrac{1}{2}$ and $(i_P +1 ) -\tfrac{1}{2}$ and $Q^2$ corresponds to a cup connecting $(j_P +1)-\tfrac{1}{2}$ and $j_Q + \tfrac{1}{2}$. This implies that
\begin{eqnarray*}
{\sf sgn}(\alpha, \la) + {\sf sgn}(\nu , \alpha) &=& {\sf sgn}(\beta , \la) + {\sf sgn}(\nu , \beta) \\ &=& 
\tfrac{1}{2}({i_Q + i_P - 1}) + \tfrac{1}{2}({j_P + 1 + j_Q})  = \tfrac{1}{2}({i_P + j_P}) +  \tfrac{1}{2}({i_Q + j_Q})\\
 &=& {\sf sgn}(\la, \mu) + {\sf sgn}(\nu , \mu).
 \end{eqnarray*}
 
 Thus we can restrict our attention to the incontractible case as the general case will follow directly by applying the dilation homomorphism. 
So we can assume that $\mu=(c^r)$ is a rectangle for $r,c>1$ and $b(P)=1$ and $b(Q)=2$. 
For the $r=2=c$ case, $\la=(2,1)$, $\mu=(2^2)$, $\nu=(1)$  and we can choose $\alpha=(2) $ and $\beta=(1^2)$. 
Here we have that $\underline{\nu}\nu \overline{\mu} \cdot \underline{\mu}\la \overline{\la}$ is given by
$$    \begin{minipage}{2.4cm}
 \end{minipage}
$$ where the last two terms are equal to $\underline{\nu} \nu \overline{\alpha}\cdot \underline{\alpha}\alpha\overline{\la}$ and $\underline{\nu}\nu \overline{\beta}\cdot \underline{\beta} \beta \overline{\la}$ as required.

Now suppose $r,c>2$.  
We can safely ignore vertical strands on the left or right as they will not affect the multiplication.
In other words, we may assume in our pictures that $r=c$.
We have that $\la=(r^{r-1},r-1)$ and $\nu=(r^{r-2},r-1,r-2)$.
We set          $\alpha=(r^{r-1},r-2)$ and  $\beta= (r^{r-2},(r-1)^2)$.  
 Here we have that $\underline{\nu} \nu \overline{\mu}\cdot \underline{\mu}\la\overline{\la}$ is given by
 
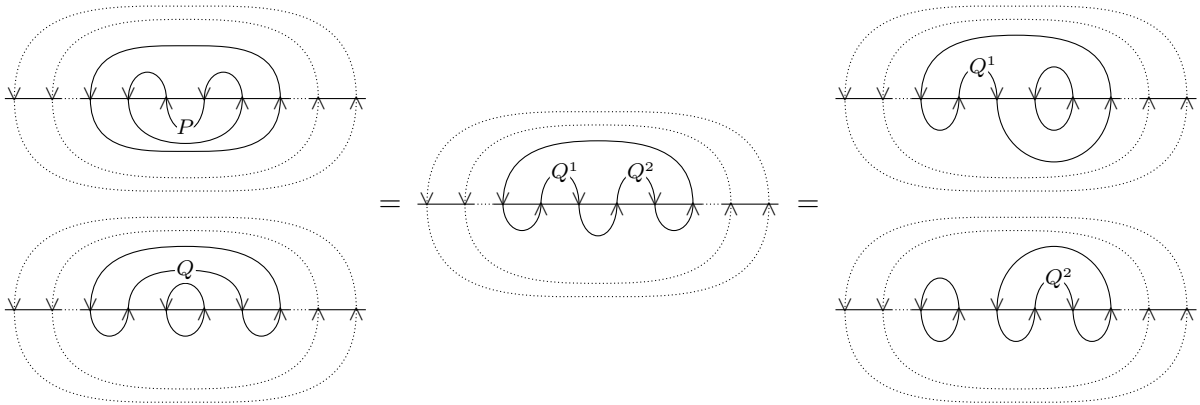
\begin{equation*} 
   \begin{minipage}{5.75cm}
 \end{minipage}
 \end{equation*}
 with $r-3$  dotted anti-clockwise circles in each diagram;     the equality follows from  applying  the $1\otimes 1 \mapsto 1$ rule a total of $r$ times as required.
Note that the last product is $\underline{\nu} \nu \overline{\alpha}\cdot \underline{\alpha}\alpha\overline{\la}$ as required. It can be shown similarly that it is also equal to $\underline{\nu} \nu \overline{\beta}\cdot \underline{\beta}\beta\overline{\la}$.
The remaining cases (where $r=2$ and $c>2$ or vice versa) are analogous, except that we make use of the oriented strand rules in the surgery procedure.
  \end{proof}

     \begin{prop}[The adjacent relation]
Let $P\in {\rm DRem}(\mu)$ and $Q\in {\rm DRem}(\mu - P)$ be adjacent.
 Let $\langle P\cup Q \rangle_\mu$ denote the smallest removable Dyck path of $\mu$ containing $P\cup Q$ (if it exists), and set $\la=\mu-P$, $\nu=\la-Q$, and $\alpha=\mu-\langle P\cup Q \rangle_\mu$.
 Then we have 
%
%
%
%
 \begin{equation}\label{thisoney2}
  i^{\sgnlamu+\sgnnula}
  \underline{\mu} \la  \overline{\la}
  \cdot  \underline{\la} \nu  \overline{\nu} 
  =
\begin{cases}
i^{2b(\langle P\cup Q\rangle_\mu) - 2b(Q)  
  +{\sf sgn}(\alpha , \mu)+{\sf sgn}(\alpha , \nu)
 } 
  \underline{\mu} \alpha  \overline{ \alpha}
  \cdot  \underline{ \alpha}  \alpha \overline{\nu} & \text{if $\langle P\cup Q \rangle_\mu$ exists,} \\
 0 & \text{otherwise.}
\end{cases}
\end{equation}

\end{prop}

\begin{proof}
We start by proving that the signs on both sides of the equation are equal. Write $\langle P\cup Q\rangle_\mu = Q'\sqcup P \sqcup Q$ and assume that $Q'$ is to the left of $Q$. Then we have that the cup corresponding to $P$ connects $i_P - \tfrac{1}{2}$ and $j_P + \tfrac{1}{2}$, the cup corresponding to $Q$ connects $(j_P +1)-\tfrac{1}{2}$ and $j_Q +\tfrac{1}{2}$, the cup corresponding to $Q'$ connects $i_{Q'}-\tfrac{1}{2}$ and $(i_P-1) + \tfrac{1}{2}$ and finally the cup corresponding to $\langle P\cup Q\rangle_\mu$ connects $i_{Q'} - \tfrac{1}{2}$ and $j_Q + \tfrac{1}{2}$ for some $i_{Q'} < i_P \leq j_P < j_Q$. Note that 
$$b(\langle P\cup Q\rangle_\mu) =
\tfrac{1}{2}(
{j_Q - i_{Q'}})+1 \quad \text{and} \quad b(Q) = 
\tfrac{1}{2}({j_Q - (j_P +1)}) +1.$$
Thus we have 
\begin{eqnarray*}
&& 2b(\langle P\cup Q\rangle_\mu) - 2b(Q)  
  +{\sf sgn}(\alpha , \nu)+{\sf sgn}(\alpha , \nu) \\
  && = 2 \left( \frac{j_Q - i_{Q'}}{2} +1  \right) -2 \left(\frac{j_Q - (j_P +1)}{2} +1 \right) + \frac{i_Q' + j_Q}{2} + \frac{i_{Q'} + (i_P -1)}{2} \\ 
  && = \frac{j_P+i_P}{2} + \frac{j_Q + (j_P+1)}{2} \\
  && = {\sf sgn}(\la, \mu) + {\sf sgn}(\nu, \la)
  \end{eqnarray*}
  as required.
  Thus we can restrict our attention to the incontractible case as the general case will follow directly by applying the dilation homomorphism. 

So we can assume that $\mu = (c^r), \la = (c^{r-1}, c-1)$ and $\nu = (c^{r-1}, c-2)$ (the case $\nu = (c^{r-2}, (c-1)^2)$ is similar). Note that $\langle P\cup Q\rangle_\mu$ exists precisely when $r\geq 2$. When $r=1$ it's easy to check that $\underline{\mu}\la \overline{\la} \cdot \underline{\la} \nu \overline{\nu} =0$ using $y\otimes y \mapsto 0$ as we are applying surgery to two $\up$-oriented propagating strands. We now assume that $r\geq 2$. Then we have $\alpha = (c^{r-2}, c-1, c-2)$. In this case, for $r\geq 3$, the equation $\underline{\mu}\la \overline{\la} \cdot \underline{\la} \nu \overline{\nu} = \underline{\mu}\alpha \overline{\alpha} \cdot \underline{\alpha}\alpha \overline{\nu}$ becomes
 \begin{equation}\label{eqna21}
  \begin{minipage}{5.5cm}
 \end{minipage} \end{equation}
 where there are $r-3$ concentric dotted outer circles  in the top and bottom diagrams on each side of  the product.  
  We can rewrite both sides of this equation   using the $1\otimes 1 \mapsto 1$ surgery procedure $r-3$ times on the dotted strands (trivially turning each pair  of dotted circles into a large dotted circle) followed by 3 times on the solid strands in order to obtain  in both cases the diagram
  \begin{equation}\label{eqna22}
 \begin{minipage}{6.5cm}
 \end{minipage}
 \end{equation}as required.
 The other cases (where $r=2$ and $c>2$ or vice versa, and where $r=c=2$) are similar. 
  \end{proof}

We thus conclude that the map $\Psi$ is indeed a $\ZZ$-graded 
homomorphism of $\Bbbk$-algebras (for $\Bbbk$ an integral domain containing a square root of $-1$).  It remains to check that this map is a bijection, we now verify this by showing that the image of 
the light leaves  basis of $\mathcal{H}_{m,n}$ is a  basis of 
$\mathcal{K}_{m,n}$ (we do this by showing that every basis element of $\mathcal{K}_{m,n}$ can be written as a product of degree 1 elements).  
Over a field, we could deduce the following result from the known Koszulity of $\mathcal{K}_{m,n}$ (this is the main result of  \cite{MR2600694}).  However, we wish to work over more general rings (where Koszulity does not hold). From an aesthetic point of view, 
we also prefer to  deduce that the algebra is generated in degree 1  directly (as appealing to  the strong cohomological property of Koszulity is somewhat  
``using a sledgehammer to crack a nut").

\begin{prop}
The map $\Psi$ is an isomorphism of graded $\Bbbk$-algebras. \end{prop}

\begin{proof}
 We have already seen  above 
 that the map is a $\Bbbk$-algebra homomorphism. 
 We need to show that the map is bijective. 
Clearly the map is bijective when one restricts attention  to the degree 0 elements 
(indexed by weights/partitions) and degree 1 elements 
(indexed by Dyck pairs of degree 1) 
of both algebras. 
We now consider elements of higher degree: we will do this in two stages.  
We first show that every element of the form $\underline{\mu}\la\overline{\la}$ is in the image, by constructing it as a product of the degree 1 elements.
We will then show that 
\begin{align}\label{profoftheform}
(\underline{\mu}\la\overline{\la})
(\underline{\la}\la\overline{\nu})
=
\underline{\mu}\la\overline{\nu}+\sum_{\zeta < \la} 
\underline{\mu}\zeta \overline{\nu}
\end{align}
and so deduce that the map is bijective by unitriangularity.

\smallskip
\noindent
{\bf Step 1:}
 We fix a sequence 
 $\mu =\la^{(0)}\supset \la^{(1)}\supset \dots\supset \la^{(m)} =\la$
 such that the Dyck path $\la^{(k-1)} = \la^{(k)} - P^k$ where $P^k$
  is a removable Dyck path of
  $ \la^{(k)}$  of  breadth $p_k$ with $p_0\geq p_1\geq p_2 \dots \geq p_m$ (this decreasing size condition is not strictly necessary, but is helpful for visualisation).
We have that 
$$
(\underline{\mu}\la \overline{\la})=
(\underline{\mu}\mu \overline{\la^{(1)}})
(\underline{\la^{(1)}}\la^{(1)} \overline{\la^{(2)}})
\dots 
(\underline{\la^{(m-1)}}\la^{(m-1)} \overline{\la} )
$$
by repeated  application of \cref{local-surg}.

\smallskip
\noindent
{\bf Step 2:}
 We first observe that the subdiagram
 $\la\overline{\la} 
 \underline{\la}\la$ within the wider diagram 
 $(\underline{\mu}\la\overline{\la})
(\underline{\la}\la\overline{\nu})$
is an oriented Temperley--Lieb diagram whose arcs are all 
anti-clockwise oriented and which is invariant under the flip through the horizontal axis. 
We will 
 work through the  cases of the  surgery rules of \cref{Khovanov  arc algebras}
 and how they can potentially be applied to the
anti-clockwise arcs in  $\la\overline{\la} 
 \underline{\la}\la$  (within the wider diagram  $(\underline{\mu}\la\overline{\la})
(\underline{\la}\la\overline{\nu})$) in turn: 
we will show that none of 
$x \otimes x \mapsto 0$, 
$x \otimes y \mapsto 0$, or 
$y \otimes y \mapsto 0$  can occur 
(thus $(\underline{\mu}\la\overline{\la})
(\underline{\la}\la\overline{\nu})\neq0$)
and that in all the other cases
there is a single term with weight $\la$ (as required).
%
%
%

We first require some notation.  
Let  $\la\in \mptn$
and   fix  $i<j\in \ZZ+\tfrac{1}{2}$  connected by a cup in $\underline{\la}$, we denote this cup by $\underline{c}
 \in \underline{\la}$ and similarly   
 we denote the reflected cap by $\overline{c}
 \in \overline{\la}$.
 We will also need to speak of the interval $\underline{I}$ on the top weight line
 (respectively $\overline{I}$ on the bottom weight line) 
  lying strictly 
 between the points $i<j\in \ZZ+\tfrac{1}{2}$.
  The surgery procedure is not concerned with the local orientation of 
the  arcs  
$\underline{c}$ and $\overline{c}$ 
 (which are always anti-clockwise in our proof), but rather the 
global orientation of the circle/strand to which the arcs $\underline{c}$ and $\overline{c}$  belong. 
Determining a global orientation of clockwise/anti-clockwise 
is  equivalent to assigning an ``inside"/``outside" label to the
 regions $\underline{I}$  and $\overline{I}$  
 in a manner we shall make more precise (case-by-case) and
this will allow us to show that certain cases cannot occur (for  topological reasons).  
We will 
assume that all our diagrams  in this proof have the    minimum number of propagating lines --- this does not affect the surgery procedure, which is defined topologically, but allows us to speak of the `left'  and `right'  propagating lines in a given circle.  
   This is illustrated in \cref{cup-cap}.

\begin{figure}[ht!]
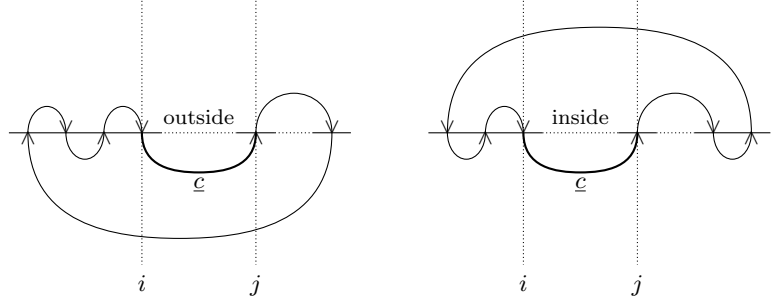

  \begin{equation*} 
 \quad   \begin{minipage}{4.75cm}
 \end{minipage}
 \end{equation*} 
\caption{ On the left we picture an anti-clockwise oriented cup $\underline{c}$ as part of a wider 
clockwise oriented circle, notice that the region $\underline{I}$ is outside of the circle.
On the right we picture an anti-clockwise oriented cup $\underline{c}$ as part of a wider 
anti-clockwise oriented circle, notice that the region $\underline{I}$ is inside of the circle.
Similar pictures can be drawn for the caps.}
\label{cup-cap}
\end{figure}

We first consider the rules in which our pair of arcs belong to the same 
connected component (either a circle or a strand).

 If the arcs  $\underline{c}$ and $\overline{c}$  both belong to the same  clockwise oriented circle, then
the regions $\underline{I}$  and $\overline{I}$  both lie {\em outside} this circle. 
 If the arcs  $\underline{c}$ and $\overline{c}$  both belong to the same anti-clockwise oriented circle, then
the regions $\underline{I}$  and $\overline{I}$ both lie {\em inside} this circle.
This is depicted in \cref{clocker}.
  Applying the surgery procedure in the former case, we obtain two non-nested  clockwise oriented circles.  The rightmost propagating strand of one circle goes through the point $i$ (which was $\down$-oriented in $\la$) 
  and the leftmost propagating strand of the other circle goes through the point $j$ (which was $\up$-oriented in $\la$) and so   $\la$ is preserved.  
  Applying the surgery procedure in the latter case, we obtain two nested  circles   
  (and sum over the $1\otimes x$ and $x\otimes 1$  orientations). 
  In the rightmost case in \cref{clocker}, the $\la$ weight corresponds to orienting the inner circle clockwise, see \cref{clocker2} (the leftmost  case in \cref{clocker} corresponds to orienting the inner circle anti-clockwise).

\begin{figure}[ht!]
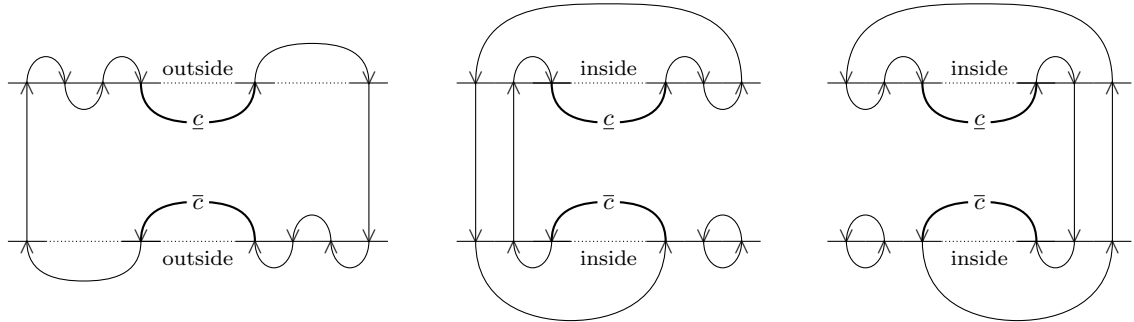

  \begin{equation*} 
 \quad   \begin{minipage}{4.75cm}
 \end{minipage}
 \end{equation*} 
\caption{ On the left we picture a pair of  anti-clockwise oriented cup/caps as part of a wider 
clockwise oriented circle, notice that the regions $\underline{I}$  and 
 $\overline{I}$ 
 lie outside of the circle.
In the two rightmost diagrams 
 we picture  the two distinct ways  that a  pair of  anti-clockwise oriented cup/caps can be part of a wider 
clockwise oriented circle, notice that the regions $\underline{I}$  and 
 $\overline{I}$ 
 lie inside of the circle.
}
\label{clocker}
\end{figure}

\begin{figure}[ht!]
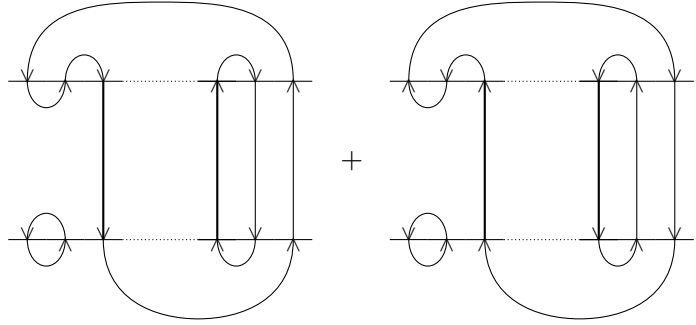

       \begin{minipage}{4.25cm}
 %
 \end{minipage}
     \caption{The effect of applying the surgery to the rightmost diagram in \cref{clocker}. 
The first term has a clockwise inner circle and an outer anti-clockwise circle and its weight is equal to $\la$. The second term has the opposite conventions and the weight $\zeta<\la$.
     }
     \label{clocker2}
     \end{figure}

We now suppose that the arcs  $\underline{c}$ and $\overline{c}$  both belong to the same strand. 
In which case,   the terminating points of this strand  are either both less than or equal to $i$ or both greater than or equal to $j$. 
We claim that the surgery 
$y \mapsto x\otimes y$ does not change the weight $\la$.  
 Applying the surgery in the former (respectively latter) case, 
we obtain  a circle whose leftmost intersection with the weight lines is at   the point  $j$ 
and a strand  whose rightmost intersection with the weight lines is at   the point   $i$
(respectively a circle whose rightmost intersection with the weight lines is at   the point  $i$ 
and a strand  whose leftmost intersection with the weight lines is at   the point   $j$).
A circle whose leftmost point is $\up $-oriented (respectively rightmost point is $\down $-oriented) is necessarily   clockwise oriented, and so the claim follows.

Next we consider the rules in which our pair of arcs belong to the distinct 
connected components (each of which is either a circle or a strand).

 The $1\otimes 1 \mapsto 1$, $ x\otimes 1 \mapsto x$,  and $1\otimes x \mapsto x$ rules can be checked in a similar fashion to above.  In the case of 
 $1\otimes 1 \mapsto 1$ the original diagram consists of two non-nested circles (see \cref{11>1}); there are two propagating strands in the circle produced by surgery
  and the orientation of the circle can be determined via the left/right propagating strand and checked to match $\la$ by the fact that this  left/right propagating strand passes through $i$ or $j$ with label $\down$ or $\up$ respectively. 
 In the case of  $ x\otimes 1 \mapsto x$ or $1\otimes x \mapsto x$,  the original diagram consists of two nested circles (see \cref{1x>x}); there are four propagating strands in the circle produced by surgery
  and the orientation of the circle can be determined via the leftmost/rightmost propagating strand and checked to match $\la$ by the fact that this  leftmost (respectively rightmost) propagating strand has the sign  $\up$ (respectively $\down$) given by the {\em opposite direction} to that encountered at  $i$  (respectively at  $j$).

\begin{figure}[ht!]
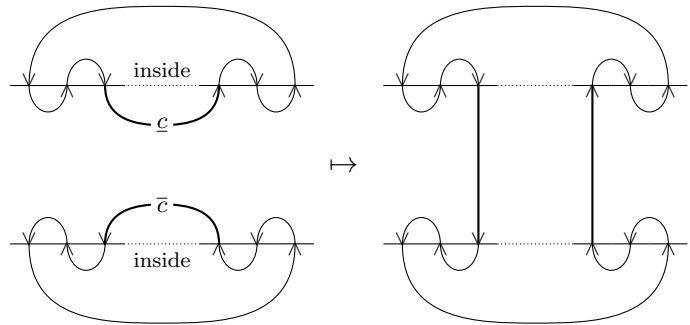

  \begin{equation*} 
   \begin{minipage}{4.1cm}
 \end{minipage}
    \end{equation*}    
     \caption{The effect of applying the $1\otimes 1 \mapsto 1$ surgery.
     }
     \label{11>1}
     \end{figure}

\begin{figure}[ht!]
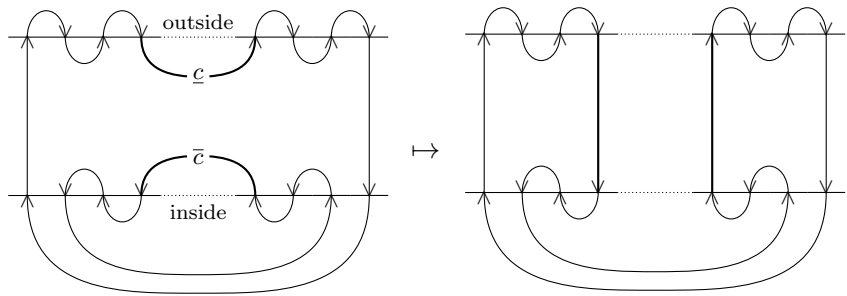

\begin{equation*} 
 \begin{minipage}{5.2cm}
 %
 \end{minipage}
       \end{equation*}
       \caption{The effect of applying the $1\otimes x \mapsto x$ surgery.}
       \label{1x>x}
       \end{figure}

 The $x \otimes x \mapsto 0$ case   is of a different flavour entirely.   We must show that we {\em never} have to apply this rule in the simplification of a product of the form above (in \cref{profoftheform}). To see this, note that the intervals $\underline{I}$ and $\overline{I}$ must both lie outside of their circles: however, 
this implies that both circles contain propagating lines and (by the planarity condition on arc diagrams) each circle must nest inside the other (as they cannot intersect), which provides  a contradiction. 
Similarly the $x\otimes y \mapsto 0$ case (also the $y \otimes y \mapsto 0$ case) gives rise to an intersecting circle and strand, which is a contradiction.

 The merging rules involving a strand can be checked in a similar fashion.
  \end{proof}

               \bibliographystyle{amsalpha}   
\bibliography{master}

 \end{document}